\documentclass{siamonline171218}

\usepackage{lipsum}
\usepackage{amsfonts,amssymb,verbatim}
\usepackage{graphicx}
\usepackage{epstopdf}
\usepackage{algorithm,algpseudocode}
\usepackage{mathtools}

\usepackage{geometry}
\usepackage[shadow,prependcaption,textsize=small]{todonotes}
\usepackage{dsfont}
\ifpdf
  \DeclareGraphicsExtensions{.eps,.pdf,.png,.jpg}
\else
  \DeclareGraphicsExtensions{.eps}
\fi

\graphicspath{{./images/}}

\usepackage[shortlabels]{enumitem}
\setlist[enumerate]{leftmargin=.5in}
\setlist[itemize]{leftmargin=.5in}

\definecolor{darkred}{rgb}{.7,0,0}

\definecolor{darkgreen}{rgb}{0,0.7,0}

\definecolor{darkblue}{rgb}{0,0,0.7}

\newcommand{\bbP}{\mathbb{P}}

\newsiamremark{remark}{Remark}
\newsiamremark{hypothesis}{Hypothesis}
\crefname{hypothesis}{Hypothesis}{Hypotheses}
\newsiamthm{claim}{Claim}
\newsiamthm{example}{Example} %
\newsiamthm{assumptions}{Assumptions} %

\usepackage{amsopn}

\newcommand{\be}{\begin{equation}}
\newcommand{\ee}{\end{equation}}
\newcommand{\beqas}{\begin{eqnarray*}}
\newcommand{\eeqas}{\end{eqnarray*}}

\newcommand{\dee}{\mathrm{d}}
\newcommand{\mC}{\mathsf{C}}
\newcommand{\mGamma}{\mathsf{\Gamma}}
\newcommand{\mI}{\mathsf{I}}
\newcommand{\iid}{\overset{\mathrm{iid}}{\sim}}
\newcommand{\eqd}{\overset{\mathrm{d}}{=}}
\newcommand{\cX}{\mathcal{X}}

\newcommand{\JC}{\mathsf{J}_{\mathsf{C}}}
\newcommand{\JNC}{\mathsf{J}_{\mathsf{NC}}}
\newcommand{\JE}{\mathsf{J}_{\mathsf{E}}}

\newcommand{\IC}{\mathsf{I}_{\mathsf{C}}}
\newcommand{\INC}{\mathsf{I}_{\mathsf{NC}}}

\newcommand{\JCN}{\mathsf{J}_{\mathsf{C}}^{N,\gamma}}
\newcommand{\JNCN}{\mathsf{J}_{\mathsf{NC}}^{N,\gamma}}
\newcommand{\JEN}{\mathsf{J}_{\mathsf{E}}^{N,\gamma}}

\newcommand{\gN}{{\gamma_N}}
\newcommand{\JCNN}{\mathsf{J}_{\mathsf{C}}^{N,\gN}}
\newcommand{\JNCNN}{\mathsf{J}_{\mathsf{NC}}^{N,\gN}}
\newcommand{\JENN}{\mathsf{J}_{\mathsf{E}}^{N,\gN}}

\newcommand{\thetaCN}{\theta_{\mathsf{C}}^{N}}
\newcommand{\thetaNCN}{\theta_{\mathsf{NC}}^{N}}
\newcommand{\thetaEN}{\theta_{\mathsf{E}}^{N}}

\newcommand{\wZ}{\bar Z}
\newcommand{\wX}{\bar X}
\newcommand{\wmu}{\bar \mu}

\newcommand{\eqc}{\propto}

\def \N {{\mathbb N}}
\def \R {{\mathbb R}}

\def \P {{\mathbb P}}

\def \la {{\langle}}
\def \ra {{\rangle}}

\def \eps {{\varepsilon}}

\DeclareMathOperator*{\argmin}{arg\,min}

\numberwithin{theorem}{section}

\newcommand{\TheTitle}{Hyperparameter Estimation in Bayesian MAP
Estimation:\\Parameterizations and Consistency} 
\newcommand{\TheAuthors}{M. M. Dunlop, T. Helin and A. M. Stuart}

\headers{Hyperparameter Estimation in Bayesian Inverse Problems}{\TheAuthors}

\title{{\TheTitle}}

\author{
  Matthew M. Dunlop\thanks{Courant Institute of Mathematical Sciences, New York University, New York, New York, 10012, USA (\email{matt.dunlop@nyu.edu}).}
  \and
  Tapio Helin\thanks{School of Engineering Science, Lappeenranta University of Technology, Lappeenranta, 53850,
Finland (\email{tapio.helin@lut.fi}).} 
  \and
  Andrew M. Stuart\thanks{Computing \& Mathematical Sciences, California Institute of Technology, Pasadena, California, 91125, USA  (\email{astuart@caltech.edu}).}
}

\ifpdf
\hypersetup{
  pdftitle={\TheTitle},
  pdfauthor={\TheAuthors}
}
\fi

\begin{document}

\maketitle

\begin{abstract}
The Bayesian formulation of inverse problems is attractive for three primary reasons: it provides a clear modelling framework; means for uncertainty quantification; and it allows for principled learning of hyperparameters. The posterior distribution may be explored by sampling methods, but for many problems it is computationally infeasible to do so. In this situation maximum a posteriori (MAP) estimators are often sought. Whilst these are relatively cheap to compute, and have an attractive variational formulation, a key drawback is their lack of invariance under change of parameterization. This is a particularly significant issue when hierarchical priors are employed to learn hyperparameters. In this paper we study the effect of the choice of parameterization on MAP estimators when a conditionally Gaussian hierarchical prior distribution is employed. Specifically we consider the centred parameterization, the natural parameterization in which the unknown state is solved for directly, and the noncentred parameterization, which works with a whitened Gaussian as the unknown state variable, and arises when considering dimension-robust MCMC algorithms; MAP estimation is well-defined in the nonparametric setting only for the noncentred parameterization. However, we show that MAP estimates based on the noncentred parameterization are not consistent as estimators of hyperparameters; conversely, we show that limits of finite-dimensional centred MAP estimators are consistent as the dimension tends to infinity. We also consider empirical Bayesian hyperparameter estimation, show consistency of these estimates, and demonstrate that they are more robust with respect to noise than centred MAP estimates. An underpinning concept throughout is that hyperparameters may only be recovered up to measure equivalence, a well-known phenomenon in the context of the Ornstein-Uhlenbeck process. 

\end{abstract}

\begin{keywords}
Bayesian inverse problems, hierarchical Bayesian, MAP estimation, 
optimization, nonparametric inference, hyperparameter inference,
consistency of estimators.
\end{keywords}

\begin{AMS}
62G05, 62C10, 62G20, 45Q05
\end{AMS}

\section{Introduction}
Let $X,Y$ be separable Hilbert spaces, and let $A:X\to Y$ be a linear map. We consider the problem of recovering a state $u \in X$ from observations $y \in Y$ given by
\begin{align}
\label{eq:data}
y = Au + \eta,\quad \eta \sim N(0,\mGamma)
\end{align}
where $\eta$ is random noise corrupting the observations. This is an example of a linear inverse problem, with the mapping $u\mapsto Au$ being the corresponding forward problem. In the applications we consider, $X$ is typically be an infinite-dimensional space of functions, and $Y$ a finite-dimensional Euclidean space $\R^J$.

Our focus in this paper is on the Bayesian approach to this inverse problem.
We view $y,u,\eta$ as random variables, assume that $u$ and $\eta$
are {\em a priori} independent with known distributions $\bbP(\dee u)$ 
and $\bbP(\dee\eta)$, and seek the posterior distribution $\bbP(\dee u|y).$
Bayes' theorem then states that
$$\bbP(\dee u|y) \propto \bbP(y|u)\bbP(\dee u).$$
In the hierarchical Bayesian approach the prior depends on hyperparameters $\theta$
which are appended to the state $u$ to form the unknown. The prior on 
$(u,\theta)$ is factored as $\bbP(\dee u,\dee\theta)=\bbP(\dee u|\theta)\bbP(\dee\theta)$
and Bayes' theorem states that
$$\bbP(\dee u,\dee\theta|y) \propto \bbP(y|u)\bbP(\dee u|\theta)\bbP(\dee\theta).$$
In this paper we study conditionally Gaussian priors in which
$\bbP(\dee u|\theta)$ is a Gaussian measure for every fixed $\theta.$

Centred methods work directly with $(u,\theta)$ as unknowns, whilst
noncentred methods work with $(\xi,\theta)$ where $u=\sqrt{C(\theta)} \xi$
and $\xi$ is, {\em a priori}, a Gaussian white noise; thus $C(\theta)$
is the covariance of $u|\theta.$ In the context of MCMC methods the
use of noncentred variables has been demonstrated to confer considerable
advantages. However the key message of this paper is 
that, when studying maximum a posteriori (MAP) estimation,
and in particular consistency of learning 
hyperparameters $\theta$ in the data-rich limit, centred parameterization is
preferable to noncentred parameterization.

\subsection{Literature Review}

The Bayesian approach is a fundamental and
underpinning framework for statistical inference \cite{berger2013statistical}. 
In the last decade it has started to become a practical computational tool for 
large scale inverse problems \cite{kaipio2006statistical}, realizing
an approach to ill-posed inverse problems introduced in the
1970 paper \cite{franklin1970well}. The subject has developing
mathematical foundations and attendant stability and approximation theories 
\cite{dashti2016bayesian,lasanen2012non,lasanen2012nonb,S10a,owhadi2015brittleness}. Furthermore,
the subject of Bayesian posterior consistency is being systematically developed 
\cite{agapiou2013posterior,agapiou2014bayesian,nickl2017bernstein,nickl2017bernsteinb,
knapik2011bayesian,knapik2013bayesian,ray2013bayesian,gugushvili2018a,gugushvili2018b}. Furthermore,
the paper \cite{knapik2016bayes} was the first to establish consistency
in the context of hyperparameter learning, as we do here, and in doing
so demonstrates that Bayesian methods have comparable capabilities
to frequentist methods, regarding adaptation to smoothness,
whilst also quantifying uncertainty. We comment further on the relationship
of our work to \cite{knapik2016bayes} in more detail later in the paper,
once the needed framework has been established.

For some problems it is still beyond reach to perform posterior sampling
via MCMC or SMC methods. For this reason maximum a posterior (MAP) estimation,
which provides a point estimator of the unknown and has a variational 
formulation, remains an important practical computational 
tool \cite{kaipio2006statistical}. Furthermore MAP estimation links
Bayesian inference with optimization approaches to inversion, and
allows for the possibility of new optimization methods informed by the
Bayesian perspective. As a consequence there is also a developing 
mathematical theory around
MAP estimation for Bayesian ill-posed inverse problems, relating to both
how to define a MAP estimator in the infinite dimensional setting
\cite{agapiou2018sparsity,clason2018generalized,dashti2013map,helin2015maximum,helin2010hierarchical}, and 
to the subject of posterior consistency of MAP estimators
\cite{agapiou2018posterior,agapiou2018rates,dashti2013map,nickl2018nonparametric,nickl2018convergence}.

The focus of this paper is hierarchical Bayesian inversion
with Gaussian priors such as the Whittle--Mat{\'e}rn and ARD priors. 
See \cite{roininen2014whittle,neal1995bayesian} for references to the literature in this
area. In this context the question of centred versus noncentred
parameterization is important in defining the problem 
\cite{papaspiliopoulos2007general}. This choice also has significant
ramifications for algorithms: there are many examples of settings in
which the noncentred approach is preferable to the centred approach
within the context of
Gibbs-based MCMC sampling \cite{agapiou2014analysis,DIS16,roberts2001inference,yu2011center} and even non-Bayesian methods such as ensemble Kalman 
inversion \cite{chada2018parameterizations}. 
Nonetheless, in the context of MAP estimation, we demonstrate in this
paper that the message is rather different: centred methods are preferable.

\subsection{Our Contribution}

The primary contributions of this paper are as follows:

\vspace{0.1in}

\begin{itemize}
\item We demonstrate that, for MAP estimation,
centred parameterizations are preferable to
noncentred parameterizations when a goal of the inference is recovery  
of the hyperparameters $\theta.$
We provide conditions on the data model and prior distribution that 
lead to theorems describing the recovery, or lack of recovery, of the true hyperparameters in the simultaneous large data/small noise limit.
\item We extend the theory to empirical Bayesian estimation of hyperparameters;
we also demonstrate additional robustness that this method has over the 
centred parameterization.
\item We demonstrate the precise sense in which hyperparameter recovery 
holds only up to measure equivalence.
\end{itemize}

\vspace{0.1in}

In \cref{sec:BIP} we introduce the Bayesian setting in which we
work, emphasizing hierarchical Gaussian priors and describing the
centred and noncentred formulations. 
In \cref{sec:PE} we review the concept of MAP estimation in a 
general setting, describing issues associated with working directly in 
infinite dimensions, and discussing different choices of parameterization. 
\Cref{sec:consistency} contains the theoretical results
concerning consistency of hyperparemeter estimation, setting up the
data-rich scenario, and studying the properties of hyperparameter
estimators for the centred, noncentred and empirical Bayes settings;
we show in particular the applicability of the theory to the case of hierarchical Whittle--Mat\'ern priors and Automatic Relevance Determination (ARD) priors.
In \cref{sec:NE} numerical results are given which illustrate
the foregoing theory. In \cref{sec:C} we conclude. 
Some lemmas required in the analysis are given in an appendix.

\section{Bayesian Inverse Problems}
\label{sec:BIP}
In this section we introduce the Bayesian hierarchical approach to the solution
of inverse problems of the form considered in the introduction. In
\cref{ssec:BT} we describe the construction of the posterior 
distribution using Bayes' theorem in both finite and infinite dimensions,
including a discussion of the Gaussian priors that we use  
when discussing consistency of MAP estimators.
\Cref{ssec:HI} is devoted to hierarchical priors, centred versus
noncentred parameterization and sampling methods associated with the 
different choice of hierarchical parameterization; this sets the context
for our results comparing MAP estimation with centred and 
noncentred parameterizations. 

\subsection{Bayes' Theorem}
\label{ssec:BT}

We describe the likelihood and posterior arising from a Bayesian
treatment of the inverse problem of interest, in the setting of Gaussian
random field priors.

\subsubsection{Gaussian Random Process Priors}
In this paper we focus on the case where the prior is (in the hierarchical
case, conditionally,) Gaussian.
Recall that probability measure $\mu_0$ on $X$ is a Gaussian measure 
if\footnote{Given a measure $\mu$ on $X$ and a measurable map $T:X\to X'$, $T^\sharp\mu_0$ denotes the \emph{pushforward measure} on $X'$, defined by $(T^\sharp\mu)(A) = \mu(T^{-1}(A))$ for all measurable $A\subseteq X'$.} $\ell^\sharp\mu_0$ is a Gaussian measure on $\R$ for any bounded linear $\ell:X\to\R$; equivalently, $\mu_0$ is a Gaussian measure if $u \sim \mu_0$ implies that $\ell(u)$ is a Gaussian random variable on $\R$ for any such $\ell$. If $X$ is a space of functions on a domain $D\subseteq\R^d$, a random variable $u$ on $X$ with law $\mu_0$ is referred to as a Gaussian process on $D$.\footnote{Also sometimes termed a Gaussian random field.} Such a Gaussian process is characterized completely by its mean function $m:D\to\R$ and covariance function $c:D\times D\to \R$:
\begin{align*}
m(x) &= \mathbb{E}^{\mu_0}(u(x))\quad\text{for all }x \in D,\\
c(x,x') &= \mathbb{E}^{\mu_0}(u(x)-m(x))(u(x')-m(x'))\quad\text{for all } x,x' \in D.
\end{align*}
Equivalently, it is characterized by its mean $m \in X$ and covariance operator $C:X\to X$,
\[
m = \mathbb{E}^{\mu_0}(u),\quad C = \mathbb{E}^{\mu_0}(u-m)\otimes(u-m).
\]
When $X = L^2(D)$, the covariance function is related to the covariance operator by
\[
(C\varphi)(x) = \int_D c(x,x')\varphi(x')\,\dee x'\quad\text{for all }\varphi \in X, x \in D,
\]
that is, $C$ is the integral operator with kernel $c$. In particular, if $C$ is the inverse of a differential operator, $c$ is the Green's function for that operator.
We now detail a number of Gaussian processes that arise as examples 
throughout the paper.

\begin{figure}
\centering
\includegraphics[width=0.32\linewidth]{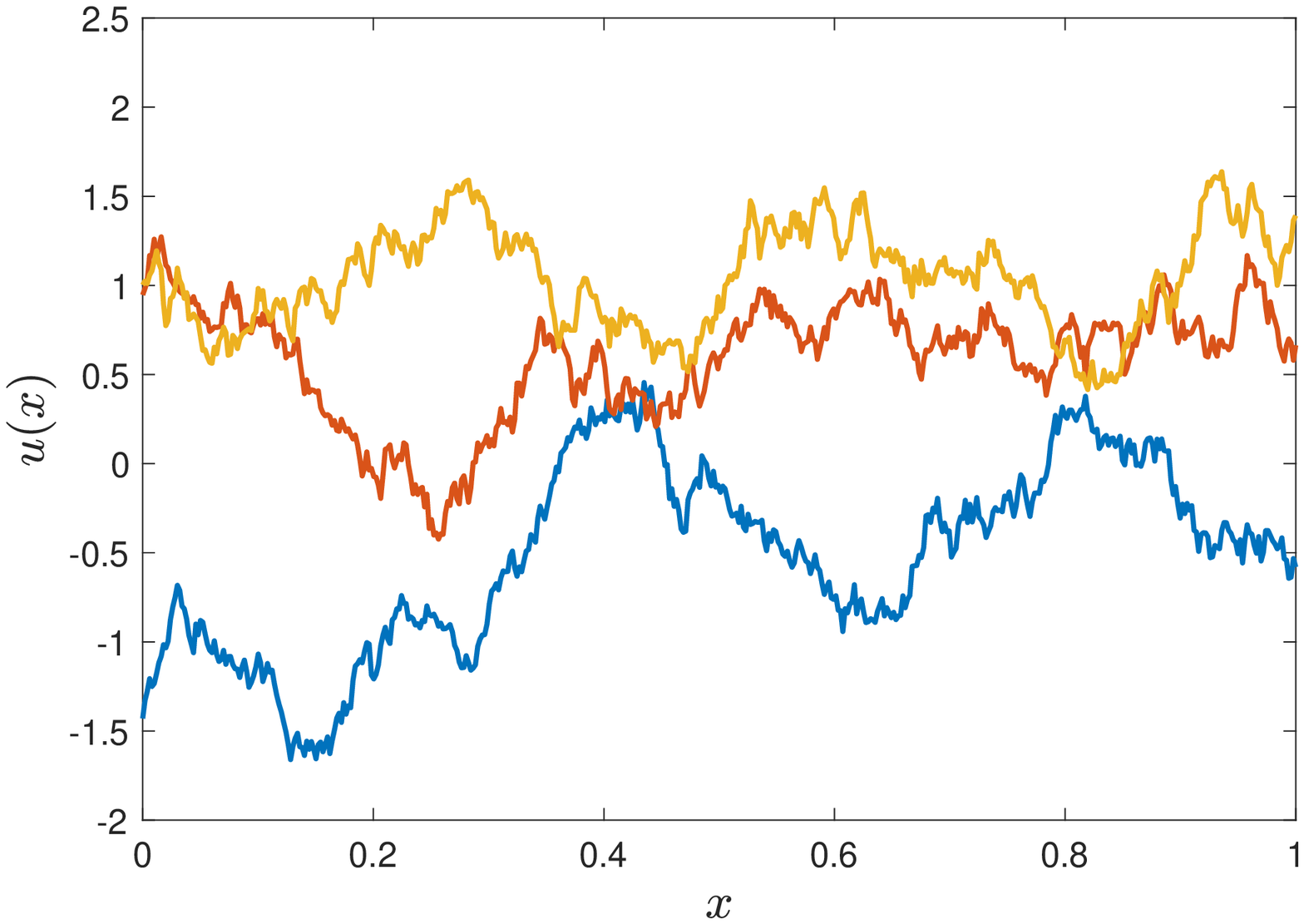}
\includegraphics[width=0.32\linewidth]{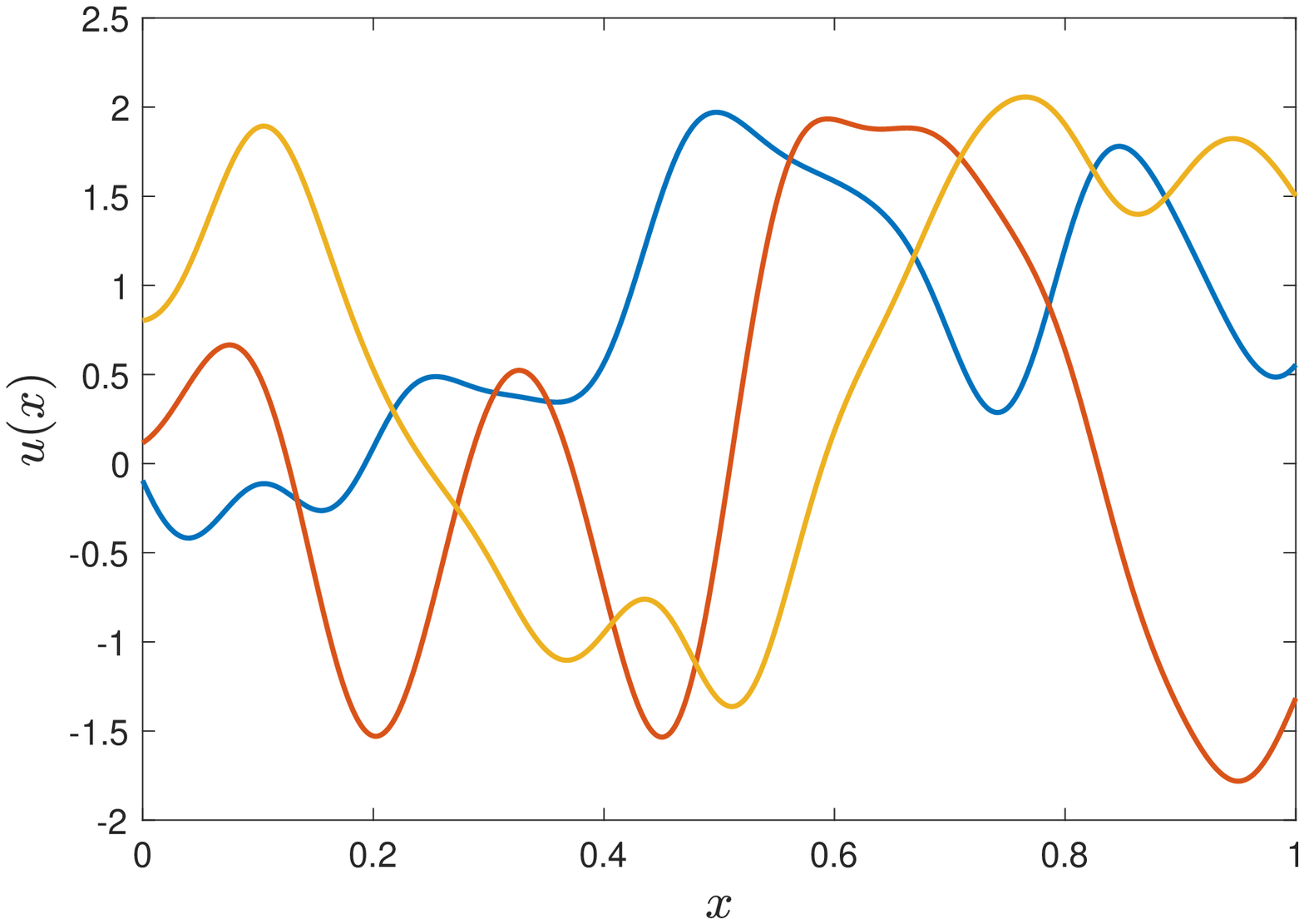}
\includegraphics[width=0.32\linewidth]{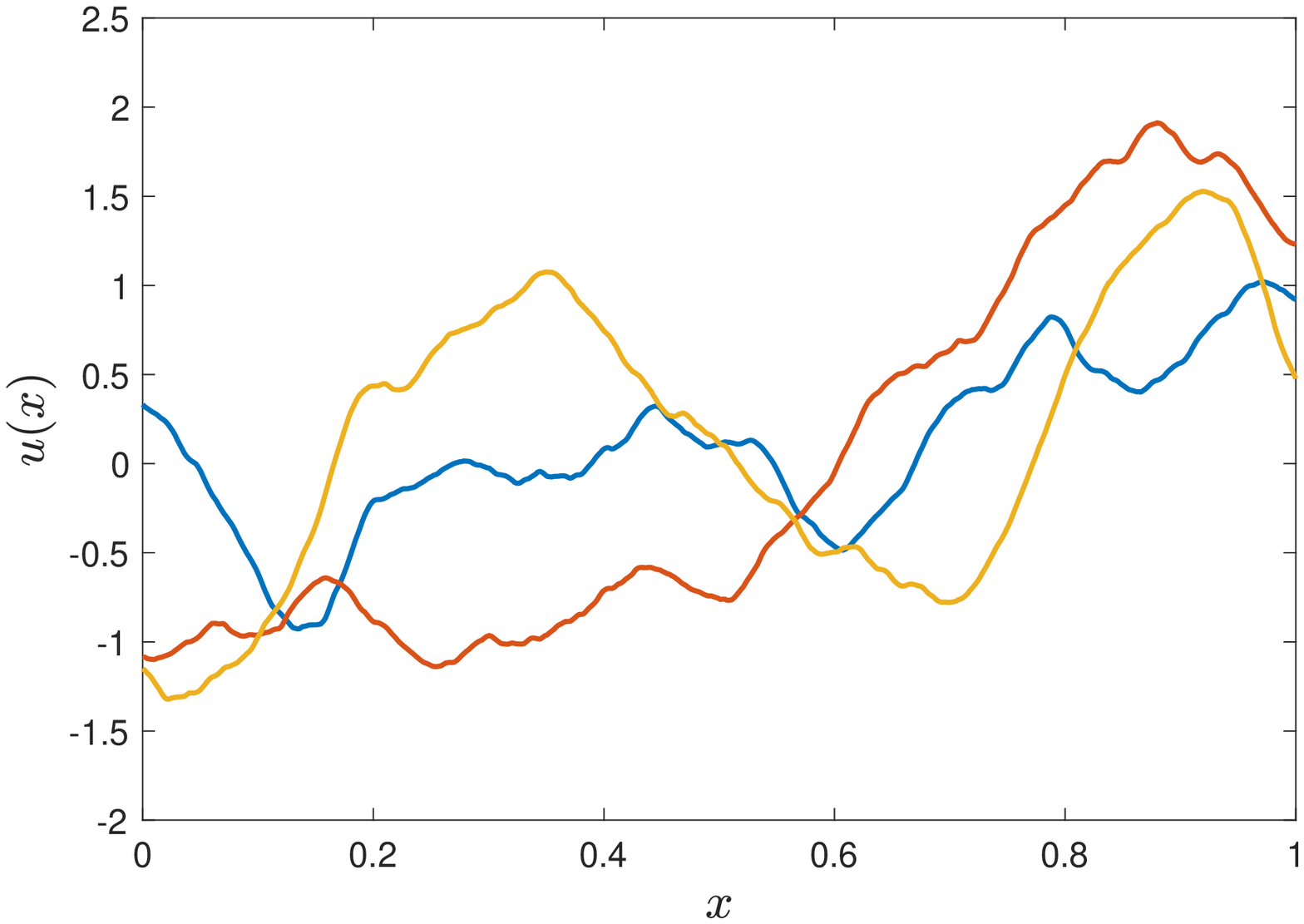}
\caption{Three sample paths each from Gaussian processes on $[0,1]$ with (left) Ornstein--Uhlenbeck, (middle) squared exponential, and (right) Matern $(\nu = 3/2)$ covariance functions.}
\label{fig:cov}
\end{figure}

\begin{example}[Ornstein--Uhlenbeck]
\label{ex:OU}
Let $D = [0,1]$. Given $\sigma,\ell > 0$, define the covariance function
\[
c_{OU}(t,t';\sigma,\ell) = \sigma^2\exp\left(-\frac{|t-t'|}{\ell}\right).
\]
This is the covariance associated with the stationary Ornstein--Uhlenbeck 
process on $[0,1]$ defined by
\[
\dee u_t = -u_t/\ell\,\dee t + \sqrt{2\sigma^2/\ell}\,\dee W_t,\quad u_0 \sim N(0,\sigma^2),
\]
where $\sigma^2$ is the variance and $\ell$ the length scale. 
The sample paths of this process are almost surely H\"older with any exponent
less than one half, everywhere in $D$.

Given observation of $u_t$ over any interval $I \subseteq D$, the diffusion 
coefficient $\sigma^2/\ell$ may be found exactly by, for example, looking at quadratic variation. To see this, we rewrite in
terms of $(\sigma,\beta) = (\sigma,\sigma^2/\ell)$,
instead of treating $(\sigma,\ell)$ as the hyperparameters. 
With this parameterization we obtain
\[
\dee u_t = -\frac{\beta}{\sigma^2} u_t\,\dee t + \sqrt{2\beta}\dee W_t,\quad u_0 \sim N(0,\sigma^2).
\]
By Girsanov's theorem, the law of $u$ is equivalent to that for 
$\sqrt{2\beta} W_t$ for any choice of $\sigma^2$. Almost sure properties
are shared between equivalent measures and for this reason it is 
possible to recover $\beta$ from observation of $u_t$ over any interval 
$I \subseteq D$, as it is from observation of $\sqrt{2\beta} W_t$
\cite{roberts2001inference}.
However joint recovery of $\beta$ and $\sigma^2$ requires more data, 
such as observation of a sample path on $[0,\infty)$; see the discussion
in \cite{vanZanten01}.

Note also that the covariance function underlying this construction
can be generalized to more general $D\subseteq\R^d$, using $|\cdot|$
to denote the Euclidean norm on $\R^d$ -- it is then typically referred 
to as the exponential covariance function.

\end{example}

\begin{example}[Squared Exponential]
Let $D \subseteq\R^d$. Given $\sigma,\ell > 0$, define the covariance function
\[
c_{SE}(x,x';\sigma,\ell) = \sigma^2\exp\left(-\frac{|x-x'|^2}{2\ell^2}\right).
\]
Then the corresponding Gaussian process has samples which are almost surely infinitely differentiable everywhere; the parameters $\sigma^2,\ell$ 
represent variance and length-scale as for the Ornstein--Uhlenbeck covariance.
\end{example}

\begin{example}[Whittle--Mat\'ern]
\label{ex:wm}
Let $D\subseteq\R^d$. The Mat\'ern (or Whittle--Mat\'ern) covariance function provides an interpolation between the previous two examples in terms of sample regularity. The parameters $\sigma^2,\ell > 0$ have the same meaning as in the
previous two examples and, additionally, we introduce the regularity parameter $\nu > 0$. Define the covariance function\footnote{Some authors may include a factor $\sqrt{2\nu}$ before the distances $|x-x'|$. We omit it here for consistency with works such as \cite{matern_spde,roininen2014whittle}, which are key to the application of results in this paper.}
\[
c_{WM}(x,x';\sigma,\ell,\nu) = \sigma^2\frac{2^{1-\nu}}{\Gamma(\nu)}\left(\frac{|x-x'|}{\ell}\right)^\nu K_\nu\left(\frac{|x-x'|}{\ell}\right)
\]
where $K_\nu$ is the modified Bessel function of the second kind of order $\nu$. Then the corresponding Gaussian process has samples which possess up to $\nu$ (fractional) weak derivatives almost surely; if the domain $D$ is suitably regular they also possess up to $\nu$ H\"older derivatives almost surely. Note that we have
\[
c_{WM}(x,x';\sigma,\ell/\sqrt{2\nu},\nu) \to
\begin{cases}
c_{OU}(x,x';\sigma,\ell) &\text{ as }\nu\to 1/2\\
c_{SE}(x,x';\sigma,\ell) &\text{ as }\nu\to \infty.
\end{cases}
\]
If $D = \R^d$, the covariance function $c_{WM}(x,x';\sigma,\ell,\nu)$ is the Green's function for the fractional differential operator given by
\begin{equation}
\label{eq:prec}
L(\sigma,\ell,\nu) = \frac{\Gamma(\nu)}{\sigma^2\ell^{d}\Gamma(\nu+d/2)(4\pi)^{d/2}}(I-\ell^2\Delta)^{\nu+d/2}.
\end{equation}
This is the precision operator for the Gaussian measure.
The corresponding covariance operator is given by $C(\sigma,\ell,\nu) = L(\sigma,\ell,\nu)^{-1}$. On more general domains $D\subseteq\R^d$, boundary conditions must be imposed on the Laplacian in order to ensure the invertibility of $L(\sigma,\ell,\nu)$; this generally affects the stationarity of samples, however conditions may be chosen such that stationarity of samples is 
(approximately) preserved \cite{daon2016mitigating,khristenko2018analysis}.

Finally, observe that if $-\Delta$ on a bounded domain $D$, subject to appropriate boundary conditions, diagonalizes with eigenbasis $\{\varphi_j\}$ and corresponding eigenvalues $\{\lambda_j\}$, then $C(\sigma,\ell,\nu)$ diagonalizes in the same basis with eigenvalues $\{\mu_j(\sigma,\ell,\nu)\}$,
\begin{align}
\label{eq:wm_eval}
\mu_j(\sigma,\ell,\nu) = \frac{\sigma^2\ell^d\Gamma(\nu+d/2)(4\pi)^{d/2}}{\Gamma(\nu)}(1+\ell^2\lambda_j)^{-\nu-d/2}.
\end{align}
This is used later when considering consistency of point estimates.
\end{example}

\subsubsection{Likelihood}
Using the model \cref{eq:data}, assuming $\eta \perp u$, we have $y|u \sim N(Au,\mGamma)$ and so
\begin{subequations}
\label{eq:landphi}
\begin{align}
\mathbb{P}(y|u) &\propto \exp(-\Phi(u;y)),\\
\Phi(u;y) &= \frac{1}{2}\|Au - y\|_\mGamma^2,\label{eq:landphib}
\end{align}
\end{subequations}
where we have introduced the notation
\[
\|z\|_\mathsf{T}^2 := \la z,\mathsf{T}^{-1}z\ra
\]
for strictly positive-definite matrix or operator $\mathsf{T}$ on
Hilbert space with inner-product $\langle \cdot, \cdot \rangle$; here
we use the Euclidean inner-product on $Y=\R^J.$ Other data models, such as those involving multiplicative or non-Gaussian noise, may lead to more complicated likelihood functions -- we focus on Gaussian additive noise in this article for both clarity of presentation and analytical tractability.

\subsubsection{Posterior}
The posterior distribution is the law of the unknown state $u$ given the data $y$, that is, the law $\P(u|y)$. Bayes' theorem shows how to construct the posterior in terms of the prior and likelihood -- at the level of densities, it says formally that
\[
\P(u|y) = \frac{\P(y|u)\P(u)}{\P(y)}.
\]
If $X = \R^n$ and the prior $\mu_0 = N(0,\mC_0)$ is Gaussian, so that the above densities exist, the posterior density is then given by
\begin{align}
\label{eq:postn}
\P(u|y) \propto \exp\left(-\Phi(u;y) - \frac{1}{2}\|u\|_{\mC_0}^2\right),
\end{align}
where the (data-dependent) proportionality constant is given by $1/\P(y)$ and 
$$\P(y) = \int_X \exp\Bigl(-\Phi(u;y)-\frac{1}{2}\|u\|_{\mC_0}^2\Bigr)\,\dee u.$$
In the setting where $X$ is a function space and $\mu_0= N(0,\mC_0)$ is 
a centred  Gaussian random process prior, the posterior measure 
$\mu^y$ is given by 
\begin{align}
\label{eq:post}
\mu^y(\dee u) \propto \exp(-\Phi(u;y))\,\mu_0(\dee u).
\end{align}
The proportionality constant is given by $1/\P(y)$ and 
$$\P(y) = \int_X \exp(-\Phi(u;y))\,\mu_0(\dee u).$$ 

The posterior itself is Gaussian in this conjugate setting:
 $\mu^y = N(m,\mC)$, where, formally,
\begin{align}
\label{eq:post_gauss}
m = \mC_0A^*(\mGamma + A\mC_0A^*)^{-1}y,\quad \mC = \mC_0 - \mC_0A^*(\mGamma + A\mC_0A^*)^{-1}A\mC_0.
\end{align}
(In infinite dimensions justification of these formulae
requires careful specification of the functional analytic setting 
\cite{lehtinen1989linear}).

In more general cases, such as when the forward map is non-linear or the prior is only conditionally Gaussian, sampling typically cannot be performed directly, and methods such as MCMC or SMC must be used instead to target the posterior. We note here that when the prior is Gaussian, MCMC and
SMC methods are available for targeting the posterior that are well-defined on function space and possess dimension-independent convergence properties \cite{beskos2008mcmc,cotter2013mcmc,beskos2015sequential} -- the existence of such methods is important when considering the choice of hierarchical parameterization in the next subsection.

\subsection{Hierarchical Inversion}
\label{ssec:HI}
The choice of a particular prior distribution with fixed parameters
may be too restrictive in practice. For example, if a Whittle--Mat\'ern Gaussian distribution is chosen, good estimates of the regularity 
parameter $\nu$ or length-scale $\ell$ may not be known, and differing 
choices of these parameters can lead to very different estimates under the posterior \cite{neal1997monte}. In the Bayesian paradigm we may
treat these parameters as unknown random variables and place a prior distribution upon them. We now describe algorithmic issues arising
from how we choose to parameterize the resulting Bayesian inverse
problem.

\subsubsection{Natural Parameterization}
We denote the hyperparameters by $\theta \in \Theta$, and assume $\Theta$ is finite-dimensional. Denoting $\rho_0$ the Lebesgue density of the prior on $\theta$, we define the conditionally Gaussian prior distribution on $(u,\theta) \in X\times\Theta$ by
\begin{align}
\label{eq:prior}
\mu_0(\dee u,\dee \theta) = \nu_0(\dee u;\theta)\rho_0(\theta)\,\dee\theta
\end{align}
where $\nu_0(\dee u;\theta) = N(0,\mC(\theta))$. Bayes' theorem is applied as above, and the posterior is now a measure on the product space $Z = X\times\Theta$:
\begin{align}
\label{eq:post_hier}
\mu^y(\dee u,\dee\theta) \propto \exp(-\Phi(u;y))\,\mu_0(\dee u,\dee\theta).
\end{align}
As in the non-hierarchical setting, it is desirable to produce samples from the posterior in order to perform inference. The posterior is no longer Gaussian even when the forward map is linear, and so we cannot sample it directly. We can however take advantage of the conditional Gaussianity of the prior and the existence of dimension-robust MCMC sampling algorithms, as outlined in \cref{alg:cmwg}.

\begin{algorithm}
\begin{algorithmic}
\caption{Centred Metropolis-within-Gibbs}
\label{alg:cmwg}
\State Choose $u^{(1)} \in X$, $\theta^{(1)} \in \Theta$.
\For{$k=1:K$}
\State Generate $u^{(k)}\mapsto u^{(k+1)}$ with a dimension-robust MCMC algorithm targeting $u|y,\theta^{(k)}$.
\State Generate $\theta^{(k)}\mapsto \theta^{(k+1)}$ with an MCMC algorithm targeting $\theta|y,u^{(k+1)}$.
\EndFor
\State\Return $\{(u^{(k)},\theta^{(k)})\}_{k=1}^K$.
\end{algorithmic}
\end{algorithm}

However, even though the update $u^{(k)}\mapsto u^{(k+1)}$ uses a dimension-robust algorithm, the update $\theta^{(k)}\mapsto\theta^{(k+1)}$ can be problematic even though it is only targeting a finite-dimensional distribution. The acceptance probability for a proposed update $\theta^{(k)}\mapsto \theta'$ involves the Radon--Nikodym derivative between the Gaussian distributions $\nu_0(\cdot;\theta^{(k)})$ and $\nu_0(\cdot;\theta')$. Such a derivative does not exist in general -- by the Feldman--Haj\`ek theorem Gaussian measures in infinite dimensions are either equivalent or singular, and the restrictive conditions required for equivalence mean that 
in many naturally occurring situations,  two Gaussian measures 
corresponding to different values of $\theta$ are singular. In
practice this means that, for \cref{alg:cmwg}, any updates 
to $\theta$ have vanishingly small acceptance probability
with respect to increasingly fine discretization of $X$; see
\cite{roberts2001inference} for a seminal analysis of this phenomenon.
In the next 
subsubsection we discuss how this problem can be circumvented by means of a reparameterization.

\subsubsection{Reparameterization}
In the natural or {\em centred parameterization} \cite{papaspiliopoulos2007general}, we treat the input $u$ to the forward map as an unknown in the problem. However, the conditional nature of the prior on the pair $(u,\theta)$ leads to sampling problems related to measure singularity as described above. We therefore look for a way of parameterizing the prior that avoids this. We first make the observation that if $\xi \sim N(0,\mI)$, then for any fixed $\theta \in \Theta$ we have
\[
\mC(\theta)^{1/2}\xi \sim N(0,\mC(\theta)) = \nu_0(\dee u;\theta).
\]
Therefore, if we choose $\xi \sim N(0,\mI)$ and $\theta \sim \rho_0$ \emph{independently}, we have
\[
(\mC(\theta)^{1/2}\xi,\theta) \sim \nu_0(\dee u;\theta)\rho_0(\dee \theta) =  \mu_0(\dee u,\dee\theta).
\]
We can hence write a sample from $\mu_0$ as a deterministic transform of a sample from the product measure $N(0,\mI)\times\rho_0$ -- this reparameterization is referred to as \emph{noncentring} in the literature \cite{papaspiliopoulos2007general}. It has the advantage that we may pass it to the posterior distribution by sampling an appropriate surrogate distribution instead of directly targeting the posterior.

We now make the preceding statement precise. Let $\wX$ be a space of distributions that white noise samples $\xi \sim N(0,\mI)$ belong to almost surely, and define the product spaces $Z = X\times\Theta$, $\wZ = \wX\times \Theta$. Define the mapping $T:\wZ\to Z$ by $T(\xi,\theta) = (\mC(\theta)^{1/2}\xi,\theta)$. Then we have the following.

\begin{proposition}[Noncentring]
Let $\mu^y$ denote the hierarchical posterior \cref{eq:post_hier} on $Z$ with prior $\mu_0$. Define the measures $\wmu_0,\wmu^y$ on $\wZ$ by $\wmu_0 = N(0,\mI)\times \rho_0$ and\footnote{Here we have implicitly extended $\Phi:X\to\R$ to $\Phi:Z\to\R$ via projection: $\Phi(u,\theta) \equiv \Phi(u)$.}
\[
\wmu^y(\dee\xi,\dee \theta) \propto \exp\big(-\Phi(T(\xi,\theta))\big)\,\wmu_0(\dee \xi,\dee\theta).
\]
Then $\mu_0 = T^\sharp \wmu_0$ and $\mu^y = T^\sharp\wmu^y$.
\end{proposition}

\begin{proof}
The first equality follows from the preceding discussion, and the second from a standard property of pushforward measures:
\[
\int f(x)\,(T^\sharp\mu)(\dee x) = \int f(T(y))\,\mu(\dee y).
\]
\end{proof}

The key consequence of this proposition is that if we sample $(\xi,\theta) \sim\wmu^y$, we have $T(\xi,\theta) \sim \mu^y$. We therefore use MCMC to target $\wmu^y$ instead of $\mu^y$ -- since the field $\xi$ and hyperparameter $\theta$ are independent under the prior, the previous measure singularity issues disappear. This leads us to \cref{alg:ncmwg}.

\begin{algorithm}
\begin{algorithmic}
\caption{Noncentred Metropolis-within-Gibbs}
\label{alg:ncmwg}
\State Choose $\xi^{(1)} \in \wX$, $\theta^{(1)} \in \Theta$.
\For{$k=1:K$}
\State Generate $\xi^{(k)}\mapsto \xi^{(k+1)}$ with a dimension-robust MCMC algorithm targeting $\xi|y,\theta^{(k)}$.
\State Generate $\theta^{(k)}\mapsto \theta^{(k+1)}$ with an MCMC algorithm targeting $\theta|y,\xi^{(k+1)}$.
\EndFor
\State \Return $\{T(\xi^{(k)},\theta^{(k)})\}_{k=1}^K$.
\end{algorithmic}
\end{algorithm}

Making the choice of noncentred variables over centred variables
leads, in the context of Gibbs-based MCMC, to significant improvement 
in algorithmic performance,
as detailed in a number of papers \cite{roberts2001inference,papaspiliopoulos2007general,yu2011center,agapiou2014analysis,CDPS18}.
However, as we demonstrate in the remainder of this paper, for
MAP estimation different considerations come in to play, and
centred methods are preferable.

\section{Point Estimation}
\label{sec:PE}
Sampling of the posterior distribution, for example using MCMC methods as mentioned in the previous section, or SMC methods as in \cite{beskos2015sequential},
may be prohibitively expensive computationally if a large number of samples are required. It is then desirable to find a point 
estimate for the solution to the problem, as opposed to the full posterior 
distribution. The conditional mean is one such point estimate, but this typically requires samples in order to be computed. Two alternative point 
estimates that we study in this paper, and define in this section, are 
the MAP estimate and the empirical Bayes (EB) estimate, both of which can be computed through optimization
procedures. The former can be interpreted as the mode of the posterior distribution, and the latter as a compromise between the mean and the mode.
In \cref{ssec:MAP} we introduce the basic MAP estimator and discuss
its properties under change of variables. In \cref{ssec:NCMAP} we
generalize to centred and noncentred hierarchical formulations; mapping
from one formulation to the other may be viewed as a hyperparameter dependent change of variables.
In \cref{ssec:EB} we define the empirical Bayes estimator.

\subsection{MAP Estimation}
\label{ssec:MAP}
In this subsection we review the definition of a MAP estimator in infinite 
dimensions and discuss its dependence on choice of parameterization.

\subsubsection{Non-Hierarchical Problems}
Suppose first that $X = \R^n$ and the posterior admits a Lebesgue density:
\[
\mu^y(\dee u) \propto \pi^y(u)\,\dee u
\]
A MAP estimate, or mode of the posterior distribution, is then any point $u \in X$ that maximizes $\pi^y$. Equivalently, it is any point that minimizes $-\log\pi^y$, which is usually more stable to deal with numerically. When the prior $\mu_0 = N(0,\mC_0)$ is taken to be Gaussian and the data 
model \cref{eq:data}, \cref{eq:landphi} is used, so that the posterior density is given by \cref{eq:postn}, a point $u \in X$ is hence a MAP estimator if and only if it minimizes the functional
\begin{align}
\label{eq:om_nonhier}
\mI(u) = \Phi(u;y) + \frac{1}{2}\|u\|_{\mC_0}^2.
\end{align}
The existence of a Lebesgue density is central to this definition of MAP estimate. We, however, are primarily interested in the case that $X$ is infinite-dimensional and a more general definition is therefore required. Dashti et al. \cite{dashti2013map} introduced such a generalization as follows.
\begin{definition}
Let $\mu$ be a Borel probability measure on a Banach space $\cX$, and denote by $B_\delta(u)$ the ball of radius $\delta$ centred at $u \in \cX$. A point $u_* \in \cX$ is said to be a MAP estimator for the measure $\mu$ if
\[
\lim_{\delta\to 0} \Bigg(\frac{\mu(B_\delta(u_*))}{\underset{u \in \cX}{\max}\,\mu(B_\delta(u))}\Bigg) = 1.
\]
\end{definition}
More general definitions have subsequently been introduced \cite{helin2015maximum,clason2018generalized}, but for the measures considered in this article they 
are equivalent to the definition above. If a Gaussian prior $\nu_0 = N(0,\mC_0)$ is chosen and the data model \cref{eq:data}, \cref{eq:landphi} is used so that the posterior distribution is given by \cref{eq:post}, then it is known \cite{dashti2013map} that a point $u$ is a MAP estimator if and only if it minimizes the Onsager-Machlup functional given by \cref{eq:om_nonhier};
the quadratic penalty term is the Cameron-Martin norm associated to
the Gaussian measure on Hilbert space $X$. This provides an explicit link between Bayesian and classical (Tikhonov) regularization. Note that, as distinct from the finite-dimensional case, the quadratic term in $\mI(u)$ is infinite at almost every point of the space $X$: $\mu_0(\{u \in X\,|\,\|u\|_{\mC_0}^2 < \infty\}) = 0$.
Although we have framed this discussion for the linear inverse problem
\cref{eq:data} subject to additive Gaussian noise,
it applies to the nonlinear setting, with Gaussian priors, and
$\Phi$ is simply the negative log-likelihood; however for this paper we
consider only linear inverse problems with additive Gaussian noise and
$\Phi$ is given by \cref{eq:landphib}.

\subsubsection{Parameterization Dependence}
MAP estimation makes a deep connection to classical applied
mathematics approaches to inversion via optimization and for this
reason it has an important place in the theory of Bayesian inversion.
However an often-cited criticism of MAP estimation within the statistics
community is that the methodology depends on the choice of parameterization of the model. To see this, assume again that $X = \R^n$ and that the posterior admits a Lebesgue density $\pi^y(u)$, so that the MAP estimator maximizes $\pi^y$. Suppose that we have a (smooth) bijective map $T:X\to X$, and instead write the unknown as $u = T(\xi)$ for some new coordinates $\xi$. Then the posterior in the coordinates $\xi$ is given by
\[
\bar{\pi}^y(\xi) = \pi^y(T(\xi))\times |\det(\nabla T(\xi))|,
\]
that is, for any bounded measurable $f:X\to\R$ we have
\[
\int_X f(u)\pi^y(u)\,\dee u = \int_X f(T(\xi))\bar{\pi}^y(\xi)\,\dee \xi.
\]
Due to the presence of this Jacobian determinant, the MAP estimators using the two coordinates generally differ. If there was no determinant term, we would have equivalence of the MAP estimators in the following sense, which is straightforward to verify.

\begin{proposition}
\label{prop:map_equiv}
$\xi_* \in \arg\max \pi^y(T(\cdot))$ if and only if $T(\xi_*) \in \arg\max \pi^y(\cdot)$.
\end{proposition}

It is natural to study how this issue of reparameterization
affects MAP estimators for hierarchical problems. In the previous 
section we chose a reparameterization in order to enable robust 
sampling. We show, however, that this reparameterization 
has undesirable effects on MAP estimation for hyperparameters.

\subsection{Hierarchical MAP Estimation}
\label{ssec:NCMAP}
In this subsection we extend the definition of a 
MAP estimator to the centred and noncentred hierarchical parameterization 
introduced in the previous section.

\subsubsection{Centred Hierarchical MAP Estimation}

We are interested in the case where $\mu_0$ on $Z = X\times\Theta$ is given by \cref{eq:prior}. The dependence of the covariance operator on the hyperparameter $\theta$ means that we cannot directly apply the above result for Gaussian measures to write down the Onsager-Machlup functional, as the normalization factor for the measure $\nu_0(\dee u;\theta)$ depends on $\theta$. If $X = \R^n$ is finite dimensional, we may write down the Onsager-Machlup functional as
\[
\IC(u,\theta) = \Phi(u;y) + \frac{1}{2}\|u\|_{\mC(\theta)}^2 + \frac{1}{2}\log\det \mC(\theta) - \log\rho_0(\theta).
\]
Now consider the case where $n \to \infty$ and $\R^n$ represents
approximation of an infinite dimensional space $X$. Since the 
limiting operator $\mC(\theta)$ is symmetric and compact, then
the determinant of finite dimensional approximations tends to zero
as $n\to\infty$. Additionally, the set of points for which the quadratic term is finite may depend on the hyperparameter $\theta$ -- in particular such sets for different values of $\theta$ may intersect only at the 
origin. As an example of this latter phenomenon, consider the  
Whittle--Mat\'ern process with precision operator $L$ given by
\cref{eq:prec}. The quadratic penalty term is $\langle u,Lu \rangle_X$
and, for different values of $\nu$ these correspond to different
Sobolev space penalizations. In summary, both the definition and 
optimization of the functional $\IC$ may be problematic in infinite 
dimensions; we show in what follows that this is also true for
sequences of finite dimensional problems
which approach the infinite dimensional limit.

Assuming now $X = \R^n$, if we fix $\theta \in \Theta$, then we can optimize $\IC(\cdot,\theta)$ to find $u(\theta) \in X$ such that
\[
\JC(\theta) := \IC(u(\theta),\theta) \leq \IC(u,\theta)\quad\text{for all }u \in X.
\]
In the linear setting \cref{eq:data} that is our focus, 
using \cref{eq:post_gauss}, we have
\begin{align}
\label{eq:u_theta}
u(\theta) = \mC(\theta)A^*(\mGamma + A\mC(\theta)A^*)^{-1}y.
\end{align}
We may then optimize $\JC(\cdot)$ to find $\theta_* \in \Theta$ such that
\[
\JC(\theta_*) = \IC(u(\theta_*),\theta_*) \leq \IC(u(\theta),\theta) \leq \IC(u,\theta)\quad\text{for all }u \in X,\theta \in \Theta. 
\]
The task of optimizing $\IC$ is hence reduced to that of optimizing $\JC$. In the next section we study the behaviour of minimizers of $\JC$ as the quality of the data increases.

\subsubsection{Noncentred Hierarchical MAP Estimation}

If we work with the noncentred coordinates introduced in the previous section, the joint prior measure is the independent product of a Gaussian measure on $\wX$ and with the hyperprior on $\Theta$. MAP estimators can hence be seen to be well-defined on the infinite-dimensional space $\wZ = \wX\times\Theta$, and to be equivalent to minimizers of the Onsager-Machlup functional
\[
\INC(\xi,\theta) = \Phi(\mC(\theta)^{1/2}\xi;y) + \frac{1}{2}\|\xi\|_{I}^2 - \log\rho_0(\theta).
\]
Note that if we reverse the transformation and write $(\xi,\theta) = T^{-1}(u,\theta) = (\mC(\theta)^{-1/2}u,\theta)$, we could equivalently define $\INC$ on $Z = X\times\Theta$ by
\[
\INC(u,\theta) = \Phi(u;y) + \frac{1}{2}\|u\|_{\mC(\theta)}^2 - \log\rho_0(\theta),
\]
in view of \cref{prop:map_equiv}. This is $\IC$, with the
the problematic log-determinant term subtracted.

As in the centred case, we can now fix $\theta$ and optimize $\INC(\cdot,\theta)$ over to  find $\xi(\theta) \in \wX$ such that
\[
\JNC(\theta) := \INC(\xi(\theta),\theta) \leq \INC(\xi,\theta)\quad\text{for all }\xi \in \wX.
\]
Again, in the linear setting \cref{eq:data}, 
using \cref{eq:post_gauss}, we have that $\xi(\theta)$ is given by
\[
\xi(\theta) = \mC(\theta)^{1/2}A^*(\mGamma + A\mC(\theta)A^*)^{-1}y.
\]
Note that $u(\theta) = \mC(\theta)^{1/2}\xi(\theta)$, which 
is consistent with \cref{prop:map_equiv}. However, note that $\JC\neq \JNC$: only the former has the log-determinant term, and so the MAP estimate for the hyperparameters typically differs between the two parameterizations.

\begin{remark}
To understand that the difference between $\JC$ and $\JNC$ is
related to the volume term arising from change of parameterization,
consider the case $X = \R^n$. We start with the measure
\begin{align*}
\mu(\dee u,\dee \theta) &\propto \exp(-\IC(u,\theta))\dee\theta\\
&= \exp\left(-\Phi(u;y) - \frac{1}{2}\|u\|_{\mC(\theta)}^2 - \frac{1}{2}\log\det\mC(\theta) + \log\rho_0(\theta)\right)\dee u\,\dee\theta\\
&=: f(u,\theta)\,\dee u\,\dee\theta.
\end{align*}
We make the transformation $(u,\theta) = T(\xi,\theta) = (\mC(\theta)^{1/2}\xi,\theta)$. The density in these new coordinates is now given by
\[
h(\xi,\theta) = f(T(\xi,\theta))\times|\det(\nabla T(\xi,\theta))|.
\]
The Jacobian determinant may be calculated as
\[
\det(\nabla T(\xi,\theta)) = \det(\nabla_\xi T_1(\xi,\theta))\det(\nabla_\theta T_2(\xi,\theta)) = \det(\mC(\theta)^{1/2})\det(\mI) = \det(\mC(\theta))^{1/2}.
\]
The log determinant terms hence cancel, giving
\[
h(\xi,\theta) \propto \exp\left(-\Phi(\mC(\theta)^{1/2}\xi;y) - \frac{1}{2}\|\xi\|_I^2 + \log\rho_0(\theta)\right) = \exp(-\INC(\xi,\theta)).
\]
\end{remark}

\subsection{Empirical Bayesian Estimation}
\label{ssec:EB}
Instead of jointly optimizing over the state $u$ and hyperparameters $\theta$, we may integrate out the state to obtain a measure just on $\theta$. In this case, one considers finding the mode of the marginal measure
\[
\P(\dee \theta|y) = \int_X \mu^y(\dee u,\dee \theta) = \left(\frac{1}{\P(y)}\int_X \exp(-\Phi(u;y))\,\nu_0(\dee u;\theta)\right)\rho_0(\theta)\,\dee\theta.
\]
The corresponding functional we wish to optimize to find $\theta$ is hence given by
\begin{align}
\label{eq:JE_int}
\JE(\theta) = -\log\left(\int_X \exp(-\Phi(u;y))\,\nu_0(\dee u;\theta)\right) -\log\rho_0(\theta).
\end{align}
In general the above functional cannot be written down more explicitly due to the intractability of the integral. When $X=\R^n$ is finite-dimensional, the integral may be approximated using a Monte Carlo average over samples $\{u_j\}_{j=1}^M \sim \exp(-\Phi(u;y))\,\nu_0(\dee u;\theta')$ for any fixed $\theta' \in \Theta$:
\begin{align*}
\JE(\theta) &= -\log\left(\int_X \frac{\nu_0(u;\theta)}{\nu_0(u;\theta')}\exp(-\Phi(u;y))\,\nu_0(\dee u;\theta')\right) - \log\rho_0(\theta)\\
&\approx -\log\sum_{j=1}^M \exp\left(\frac{1}{2}\|u_j\|^2_{\mC(\theta')}-\frac{1}{2}\|u_j\|_{\mC(\theta)}^2 + \frac{1}{2}\log\det\mC(\theta')\mC(\theta)^{-1}\right) - \log\rho_0(\theta)\\
&=: \JE(\theta;\theta',\{u_j\}),
\end{align*}
where the log-sum-exp trick may be used numerically to avoid underflow \cite[\S 3.5.3]{murphy}. One may then aim to approximately optimize $\JE$ via \cref{alg:em}, which alternates approximating the integral above via samples from the conditional posterior given the current hyperparameter values, and optimizing over the hyperparameters given these samples, a form
of expectation-maximization (EM) algorithm. The sampling in each step is typically be performed using a dimension-robust MCMC algorithm, such as the pCN algorithm; the resulting random sequence $\{\theta^{(k)}\}$ can then be averaged, for example, to produce a single hyperparameter estimate.

\begin{algorithm}
\caption{EM Algorithm}
\label{alg:em}
\begin{algorithmic}
\State Choose initial estimate $\theta^{(1)}$ for the hyperparameter.
\For{$k=1:K$}
	\State Sample $\{u_j^{(k)}\}_{j=1}^M \sim \exp(-\Phi(u;y))\,\nu_0(\dee u;\theta^{(k)})$.
	\State $\theta^{(k+1)} \gets \underset{\theta\in\Theta}{\mathrm{argmin}}\;\JE(\theta;\theta^{(k)},\{u_j^{(k)}\})$.
\EndFor
\end{algorithmic}
\end{algorithm}

In the linear setting \cref{eq:data}, the integral in \cref{eq:JE_int} 
can be computed analytically using Gaussian structure. Rather than calculate the integral above directly, we note that we may rewrite the data in noncentred coordinates as
\[
y = A\mC(\theta)^{1/2}\xi + \eta,\quad \eta \sim N(0,\mGamma)
\]
where $\xi \sim N(0,\mI)$; from this it can be seen that $\P(y|\theta) = N(0,\mGamma + A\mC(\theta) A^*)$. Thus, by Bayes' theorem,
\begin{align*}
\P(\dee \theta|y) \propto\frac{1}{\sqrt{\det(\mGamma + A\mC(\theta) A^*)}} \exp\left(-\frac{1}{2}\|y\|_{\mGamma + A\mC(\theta) A^*}^2\right)\rho_0(\theta)\,\dee\theta.
\end{align*}
Modes of this marginal measure are then given by minimizers of the functional
\[
\JE(\theta) = \frac{1}{2}\|y\|_{\mGamma + A\mC(\theta) A^*}^2 + \frac{1}{2}\log\det(\mGamma + A\mC(\theta) A^*) - \log\rho_0(\theta).
\]
Despite involving norms and determinants on the data space rather than the state space, the form of $\JE$ is actually very similar to that of $\JC$, as is
shown in the following section.

\begin{remark}
In the spirit of this paper, we later consider the mode of $\P(\dee \theta| y)$ as the empirical estimator for $\theta$. Such a choice can also be considered as a regularized maximum likelihood estimator, where the hyperparameter density acts as a regularizer. 
\end{remark}

\section{Consistency of Point Estimators}
\label{sec:consistency}
In the previous section we derived three different functionals, $\JC,
\JNC$ and $\JE$. Optimizing each of these functionals 
leads to different estimates of the hyperparameters of the same underlying
statistical model. In this section we study the behaviour of these estimates 
in a data-rich scenario. In \cref{ssec:datamodel} we spell
out the precise data model that we use; it corresponds to a finite
dimension $N$ truncation of the linear inverse problem \cref{eq:data},
and since subsequent limit theorems focus on the situation in
which the observational noise standard deviation
$\gamma$ is small, we write the resulting functionals to be
optimized as $\JCN, \JNCN$ and $\JEN$. 
\cref{prop:functionals} gives the exact form for
the resulting functionals and demonstrates the similar form taken
by $\JCN$ and $\JEN$, whilst also showing that $\JNCN$ is substantially
different.
\Cref{ssec:CoM} contains
the limit theorems which characterize the three different estimators
in the data-rich limit. \cref{thm:main} shows that
the centred and empirical Bayes approaches recover the true
parameter value whilst the noncentred approach does not.
In \cref{ssec:E} we discuss examples.  

\subsection{The Data Model}
\label{ssec:datamodel}
In order to analyse the behaviour of these minimizers, we work in the simplified setup where the forward map $A$ is linear, and $A^*A$ is simultaneously diagonalizable with the family of covariance operators. Specifically, we make the following assumptions.
\begin{assumptions}
We assume in what follows that:
\begin{enumerate}[(i)]
\item The map $A^*A$ and family of prior covariance operators $\{\mC(\theta)\}_{\theta \in \Theta}$ are strictly positive and simultaneously diagonalizable with orthonormal eigenbasis $\{\varphi_j\}$, and we have
\[
A^*A\varphi_j = a_j^2\varphi_j,\quad C(\theta)\varphi_j = \mu_j(\theta)\varphi_j\quad\text{for all }j\in \N, \theta \in \Theta.
\]
\item The noise covariance $\mGamma = \gamma^2 \mI$ is white.
\end{enumerate}
\end{assumptions}

\begin{remark}
The second assumption is essentially equivalent to assuming that the noise covariance $\mGamma$ is non-degenerate: we may work with the transformed data $\mGamma^{-1/2}y$ and redefine $A$ as $\mGamma^{-1/2}A$. We could hence equivalently replace $A^*A$ with $A^*\Gamma^{-1}A$ in the first assumption.
\end{remark}

We choose the basis $\{\psi_j\}$ for $Y$ given by $\psi_j = A\varphi_j/\|A\varphi_j\| = A\varphi_j/a_j$; it can readily be checked that this is an orthonormal basis. Assume that the true state $u^\dagger$ that generates the data is drawn from the distribution $N(0,\mC(\theta^\dagger))$ for some $\theta^\dagger \in \Theta$. We define the data $y^\gamma \in Y$ by
\[
y^\gamma = Au^\dagger + \gamma\eta,\quad \eta \sim N(0,\mI),
\]
where we have made the dependence of the data on $\gamma$ explicit. We define individual observations $y^\gamma_j \in \R$ of the data $y^\gamma \in Y$ as
\begin{align}
\label{eq:data_c}
\notag y_j^\gamma &:= \la y^\gamma,\psi_j\ra\\
\notag &= \frac{1}{a_j}\la Au^\dagger,A\varphi_j\ra + \gamma\la \eta,\psi_j\ra\\
\notag &= \frac{1}{a_j}\la u^\dagger,A^*A\varphi_j\ra + \gamma\la \eta,\psi_j\ra\\
&=  a_j u_j^\dagger + \gamma\eta_j,\quad \eta_j\iid N(0,1),\quad j \in \N,
\end{align}
where $u_j^\dagger := \la u^\dagger,\varphi_j\ra$. It is convenient to note that we have the equality in distribution with the noncentred-type representation
\begin{align}
\label{eq:data_nc}
y_j^\gamma \eqd \sqrt{a_j^2\mu_j(\theta^\dagger) + \gamma^2}\xi_j^\dagger,\quad \xi_j^\dagger \iid N(0,1),\quad j \in \N.
\end{align}
As we establish results regarding convergence of minimizers in probability, there is no loss in generality in assuming that the data is given by \cref{eq:data_nc} instead of \cref{eq:data_c}.

The infinite collection of scalar problems \cref{eq:data_c} is equivalent to the full infinite-\linebreak dimensional problem. We consider a sequence of finite-dimensional problems arising from taking the first $N$ of these observations, so that data provided for the $N^{th}$ problem is given by
\begin{align}
\label{eq:data_c_N}
y_j^\gamma = a_j u_j^\dagger + \gamma\eta_j,\quad \eta_j\iid N(0,1),\quad j = 1,\ldots,N.
\end{align}
We take the prior distribution for these problems to be the projection of the full prior onto the span of the first $N$ eigenfunctions $\{\varphi_j\}_{j=1}^N$, so that both the state and the data are finite-dimensional. To motivate why we use this projection of the prior distribution, we look at the structure of the likelihood. Writing $y^\gamma_{1:N}$ for the vector of observations $(y^\gamma_1,\ldots,y^\gamma_N) \in \R^N$, the negative log-likelihood of $y^\gamma_{1:N}$ given $u$ takes the form
\begin{align*}
\Phi_\gamma(u;y_{1:N}^\gamma) &= \frac{1}{2\gamma^2}\sum_{j=1}^N |\la Au-y^\gamma,\psi_j\ra|^2\\
&= \frac{1}{2\gamma^2}\sum_{j=1}^N \left|\frac{1}{a_j}\la A^*Au,\varphi_j\ra - \la y^\gamma,\psi_j\ra\right|^2\\
&= \frac{1}{2\gamma^2}\sum_{j=1}^N |a_j u_j - y_j^\gamma|^2
\end{align*}
where $u_j := \la u,\varphi_j\ra$. The posterior on $u_j$ for $j > N$ is hence uninformed by the observations and remains the same as the prior. To be more 
explicit, for the $N^{th}$ problem we choose the conditional prior distribution $\nu_0^N(\cdot;\theta) = P_N^\sharp \nu_0(\cdot;\theta)$, where $P_N:X\to\R^N$ is given by $(P_Nu)_j = u_j$ for $j=1,\ldots,N$. Since $\nu_0(\cdot;\theta) = N(0,\mC(\theta))$ is Gaussian on $X$, this is equivalent to saying $\nu_0^N(\cdot;\theta) = N(0,P_N\mC(\theta)P_N^*)$ is Gaussian on $\R^N$.

We denote by $\JCN$, $\JNCN$ and $\JEN$ the functionals $\JC$, $\JNC$ and $\JE$ respectively constructed for these finite dimensional problems. We study the convergence of estimates of the hyperparameter $\theta$
to its true value $\theta^\dagger$ in the simultaneous limit of the number of observations $y^\gamma_1,\ldots,y^\gamma_N$ going to infinity and the noise level $\gamma$ going to zero. 

\begin{remark}
\label{rem:trunc}
The above truncation has no effect on the forms of the functionals $\JNCN$ and $\JEN$; $\JCN$ however does change. Nonetheless, if the non-truncated prior is used to write down $\JCN$, poor estimates for hyperparameters are obtained as the prior then dominates over the observations, see \cref{ssec:trunc} for an illustration.
\end{remark}

For brevity, in what follows we use the notation $f(\theta) \eqc g(\theta)$ to mean that  $f(\theta) = \alpha g(\theta) + \beta$ for some constants $\alpha,\beta$ -- note that $f$ and $g$ then have the same minimizers. 

\begin{proposition}
\label{prop:functionals}
Define $s_j^\gamma(\theta) = a_j^2 \mu_j(\theta) + \gamma^2$. Then we have
\begin{align}
\label{eq:jcn}\JCN(\theta) &\eqc \frac{1}{2N} \sum_{j=1}^N \left[ \frac{(y_j^\gamma)^2}{s^\gamma_j(\theta)} -  \log \frac{\mu_j(\theta^\dagger)}{\mu_j(\theta)} \right] - \frac{1}{N}\log \rho_0(\theta),\\
\label{eq:jncn}\JNCN(\theta) &\eqc \frac{1}{2N} \sum_{j=1}^N \frac{(y_j^\gamma)^2}{s^\gamma_j(\theta)} - \frac{1}{N}\log \rho_0(\theta),\\
\label{eq:jen}\JEN(\theta) &\eqc \frac{1}{2N} \sum_{j=1}^N \left[ \frac{(y_j^\gamma)^2}{s^\gamma_j(\theta)} -  \log \frac{s^\gamma_j(\theta^\dagger)}{s^\gamma_j(\theta)} \right] - \frac{1}{N}\log \rho_0(\theta).
\end{align}
\end{proposition}

\begin{remark}
We have made the shifts
\[
\JCN(\theta) \mapsto \JCN(\theta) - \frac{1}{2}\sum_{j=1}^N \log\mu_j(\theta^\dagger),\quad \JEN(\theta) \mapsto \JEN(\theta) - \frac{1}{2}\sum_{j=1}^N \log s_j^\gamma(\theta^\dagger).
\]
These do not affect minimizers, as the shifts are constant in $\theta$. These transformations are useful in the next section in the derivation of a limiting functional as $N\to\infty$ and $\gamma\to 0$.
\end{remark}

\begin{proof}
Instead of the expression for $u(\theta)$ given by \cref{eq:u_theta}, we use the alternative expression
\[
u(\theta) = (A^*\mGamma^{-1}A + \mC(\theta)^{-1})^{-1}A^*\mGamma^{-1}y = \frac{1}{\gamma^2}\left(\frac{1}{\gamma^2}A^*A + \mC(\theta)^{-1}\right)^{-1}A^*y
\]
which follows from the Sherman--Morrison--Woodbury formula. Using the simultaneous diagonalizability, we then have that
\begin{align*}
u_j(\theta) &:= \la u(\theta),\varphi_j\ra = \frac{1}{\gamma^2}\left(\frac{a_j^2}{\gamma^2} + \frac{1}{\mu_j(\theta)}\right)^{-1} \la A^*y,\varphi_j\ra = \frac{a_j\mu_j(\theta)}{s_j^\gamma(\theta)}y_j^\gamma.
\end{align*}
Now consider the functional
\[
\mathsf{J}_0^{N,\gamma}(\theta) := \Phi_\gamma(u;\theta) + \frac{1}{2}\|u\|_{\mC(\theta)}^2.
\]
We may calculate
\begin{align*}
\mathsf{J}_0^{N,\gamma}(\theta) &= \frac{1}{2\gamma^2}\sum_{j=1}^N (a_j u_j(\theta) - y_j^\gamma)^2 + \frac{1}{2}\sum_{j=1}^N \frac{u_j(\theta)^2}{\mu_j(\theta)}\\
&=\frac{1}{2}\sum_{j=1}^N (y_j^\gamma)^2\left[\frac{1}{\gamma^2}\left(\frac{a_j^2\mu_j(\theta)}{s_j^\gamma(\theta)}-1\right)^2 + \frac{a_j^2\mu_j(\theta)}{s_j^\gamma(\theta)^2}\right]\\
&= \frac{1}{2}\sum_{j=1}^N \frac{(y_j^\gamma)^2}{s_j^\gamma(\theta)^2}\left[ \frac{1}{\gamma^2}\left(a_j^2\mu_j(\theta) - s_j^\gamma(\theta)\right)^2 + a_j^2\mu_j(\theta)\right]\\
&= \frac{1}{2}\sum_{j=1}^N \frac{(y_j^\gamma)^2}{s_j^\gamma(\theta)}.
\end{align*}
The expression for $\JNCN$ then follows. For $\JCN$, we note that
\[
\frac{1}{2}\log\det\mC(\theta) = \frac{1}{2}\sum_{j=1}^N \log \mu_j(\theta) \eqc -\frac{1}{2}\sum_{j=1}^N \log \frac{\mu_j(\theta^\dagger)}{\mu_j(\theta)}
\]
from which the result follows. Finally we deal with the empirical Bayes case $\JEN$. Observe that
\begin{align*}
\frac{1}{2}\|y^\gamma\|_{\mGamma + A\mC(\theta)A^*}^2 &= \frac{1}{2}\sum_{i,j=1}^N y_i^\gamma y_j^\gamma\la \psi_i,(\mGamma + A\mC(\theta)A^*)^{-1}\psi_j\ra\\
&= \frac{1}{2}\sum_{i,j=1}^N y_i^\gamma y_j^\gamma\cdot\frac{1}{a_ia_j}\la \varphi_i,A^*(\mGamma + A\mC(\theta)A^*)^{-1}A\varphi_j\ra.
\end{align*}
Using the Sherman--Morrison--Woodbury identity again, we may write
\begin{align*}
A^*(\mGamma + A\mC(\theta)A^*)^{-1}A &= A^*\mGamma^{-1}A - A^*\mGamma^{-1}A(A^*\mGamma^{-1}A+\mC(\theta)^{-1})^{-1}A^*\mGamma^{-1}A\\
&= \frac{1}{\gamma^2}A^*A - \frac{1}{\gamma^2}A^*A\left(\frac{1}{\gamma^2}A^*A+\mC(\theta)^{-1}\right)^{-1}\frac{1}{\gamma^2}A^*A,
\end{align*}
and so by the simultaneous diagonalizability, and orthonormality of $\{\varphi_j\}$,
\begin{align*}
\frac{1}{2}\|y\|_{\mGamma + A\mC(\theta)A^*}^2 &= \frac{1}{2}\sum_{j=1}^N \frac{(y_i^\gamma)^2}{a_j^2}\left[\frac{a_j^2}{\gamma^2} - \frac{a_j^2}{\gamma^2}\left(\frac{a_j^2}{\gamma^2} - \frac{1}{\mu_j(\theta)}\right)^{-1}\frac{a_j^2}{\gamma^2}\right]\\
&= \frac{1}{2}\sum_{j=1}^N \frac{(y_i^\gamma)^2}{\gamma^2}\left[1 - \frac{a_j^2\mu_j(\theta)}{s_j^\gamma(\theta)}\right]\\
&= \frac{1}{2}\sum_{j=1}^N \frac{(y_i^\gamma)^2}{s_j^\gamma(\theta)}.
\end{align*}
To deal with the log-determinant term, we use \cref{lem:determinant} to see that
\begin{align*}
\frac{1}{2}\log\det(\mGamma + A\mC(\theta) A^*) = \frac{1}{2}\log\det(A^*\mGamma A + A^*A\mC(\theta) A^*A) - \frac{1}{2}\log\det(AA^*).
\end{align*}
Since $\{\varphi_j\}$ is an orthonormal basis for $X$, the first determinant may be calculated as 
\[
\frac{1}{2}\log\det(A^*\mGamma A + A^*A\mC(\theta) A^*A) = \frac{1}{2}\sum_{j=1}^N \log(a_j^2\gamma^2 + a_j^4\mu_j(\theta)) = \frac{1}{2}\sum_{j=1}^N \log(a_j^2s_j^\gamma(\theta))
\]
and so
\[
\frac{1}{2}\log\det(\mGamma + A\mC(\theta) A^*)  \eqc -\frac{1}{2}\sum_{j=1}^N \log\frac{s_j^\gamma(\theta^\dagger)}{s_j^\gamma(\theta)}
\]
from which the result follows.
\end{proof}

\subsection{Convergence of Minimizers}
\label{ssec:CoM}

We study convergence of the minimizers of the random functionals $\JCN, \JNCN$ and $\JEN$ in the simultaneous limit $N\to \infty$ and $\gamma\to 0$. We establish that, if the noise level decays sufficiently fast relative to the 
smallest value of the product of the singular values and the prior
covariance, for the truncated problem, then the true hyperparameter is recovered in the cases of the centred MAP and empirical Bayes estimates. We also establish that it is not recovered in the case of the noncentred MAP estimate.

Let $\gamma_N > 0$ denote the noise level when $N$ observations are taken. We define $s_j^{\gN}(\theta) = a_j^2 \mu_j(\theta) + \gN^2$ as in \cref{prop:functionals}, and define
\[
b_j^N(\theta) = \frac{s_j^\gN(\theta^\dagger)}{s_j^\gN(\theta)}.
\]
In order to establish the convergence, we make the following assumptions.

\begin{assumptions}
\label{ass:conv}
We assume in what follows that:
\begin{enumerate}[(i)]
\item $\Theta\subseteq\R^k$ is compact.
\item $\underset{j=1,\ldots,N}{\min}\,a_j^2\mu_j(\theta)/\gamma_N^2 \to \infty$ as $N\to\infty$ for all $\theta \in \Theta$.
\item $g(\theta,\theta^\dagger) := \underset{j\to\infty}{\lim} \frac{\mu_j(\theta^\dagger)}{\mu_j(\theta)}$ exists for all $\theta \in \Theta$, and the map $\theta\mapsto g(\theta,\theta^\dagger) - \log g(\theta,\theta^\dagger)$ is lower semicontinuous.
\item If $g(\theta,\theta^\dagger) = 1$, then $\theta = \theta^\dagger$.
\item The maps $\theta \mapsto \log\mu_j(\theta) $ are Lipschitz on $\Theta$ for each $j \in \N$, with Lipschitz constants uniformly bounded in $j$.
\item The maps $\theta \mapsto b_j^N(\theta)$ are Lipschitz on $\Theta$ for each $j=1,\ldots,N$, $N\in\N$, with Lipschitz constants uniformly bounded in $j,N$.
\item The map $\theta \mapsto \log\rho_0(\theta)$ is Lipschitz on $\Theta$.
\end{enumerate}
\end{assumptions}

Assumption (i) is made to avoid complications with hyperparameter estimates potentially diverging. Assumption (ii) gives the rate at which the noise must decay relative to the decay of the singular values of the (whitened) forward map -- the more ill-posed the problem is, and the weaker the prior is, the faster the noise must vanish. Assumption (iii) allows a limiting functional to be identified, and (iv) is an identifiability assumption which allows us to identify the true hyperparameter. Assumptions (v)-(vii) are made to ensure the functionals $\JCNN,\JNCNN,\JENN$ are also Lipschitz with Lipschitz constants (almost surely) uniformly bounded in $N$; note that when combined with the assumed compactness 
of $\Theta$, we thus obtain existence of minimizers of these functionals over
$\Theta.$

\begin{remark}
\label{rem:Ngamma}
Instead of having the noise level $\gamma_N$ a function of the number of observations, we could also consider having the number of observations $N_\gamma$ as a function of the noise level -- this may be more appropriate in practice as one may not have control over the noise level. In this case, one would need to replace Assumption (ii) with
\begin{align}
\label{eq:wm_cond2}
\min_{j=1,\ldots,N_\gamma} a_j^2\mu_j(\theta)/\gamma\to \infty\text{ as }\gamma\to 0
\end{align}
in order to obtain analogous results. We work with $\gamma_N$ to make the arguments clearer: our sequences of functionals are indexed by a discrete rather than continuous parameter.
\end{remark}

\begin{theorem}
\label{thm:main}
Let \cref{ass:conv} hold, and let $\{\thetaCN\}, \{\thetaEN\}, \{\thetaNCN\}$ denote sequences of minimizers over $\Theta$ of $\{\JCNN\}$, $\{\JENN\}$, $\{\JNCNN\}$ respectively. 
\begin{enumerate}[(i)]
\item $\thetaCN, \thetaEN \to \theta^\dagger$ in probability as $N\to\infty$.
\item Assume further that $g(\cdot,\theta^\dagger)$ has a unique minimizer $\theta_*$. Then $\thetaNCN \to \theta_*$ in probability as $N\to\infty$.
\end{enumerate}
\end{theorem}

\begin{remark}
Recently, Knapik et al. \cite{knapik2016bayes} studied consistency of empirical maximum likelihood estimators for inverse problems. Like us, they consider a 
diagonalizable problem, but their analysis is confined to a single 
hyperparameter describing the regularity of the Gaussian prior, and to the 
empirical Bayes procedure only, not MAP estimation. However, in their
setting they can go further than in ours. 
Their main results in \cite[Thm. 1 and 2]{knapik2016bayes} show convergence 
{\em rates} of the empirical estimator, like us in probability, and they use
this to deduce that the empirical posterior on $u$ contracts around the ground 
truth at an optimal rate. Whereas we assume data to be generated 
according to $\P(y | \theta^\dagger)$, Knapik et al. consider $\P(y | u^\dagger)$ as the data generating distribution  and the function $u^\dagger$ 
implicitly identifies the true regularity $\theta^\dagger$.
\end{remark}

\begin{remark}
In general it is the case that $\theta_*\neq\theta^\dagger$, and so the result concerning the convergence of $\{\thetaNCN\}$ is a negative result: the true hyperparameter is not recovered.
\end{remark}

\begin{proof}[Proof of \cref{thm:main}]
We establish the result in full for $\JCNN$, and note the small modifications required to establish the results for $\JENN$ and $\JNCNN$. 
We start by proving item (i). We have
\[
\JCNN(\theta) = \frac{1}{2N} \sum_{j=1}^N \left[ \frac{(y_j^{\gN})^2}{s^{\gN}_j(\theta)} - \log \frac{\mu_j(\theta^\dagger)}{\mu_j(\theta)} \right] - \frac{1}{N}\log \rho_0(\theta).
\]
We rewrite $y_j^{\gN}$ using the representation \cref{eq:data_nc}:
\[
y_j^{\gN} \eqd \sqrt{s_j^{\gN}(\theta^\dagger)}\zeta_j,\quad\zeta_j \iid N(0,1),
\]
and so 
\[
\JCNN(\theta) \eqd \frac{1}{2N} \sum_{j=1}^N \left[ b_j^N(\theta)\zeta_j^2 - \log \frac{\mu_j(\theta^\dagger)}{\mu_j(\theta)} \right] - \frac{1}{N}\log \rho_0(\theta).
\]
We can see formally from the assumptions that, for each $\theta \in\Theta$, $b_j^N(\theta) \to g(\theta,\theta^\dagger)$ as $j,N\to\infty$, and so the strong law of large numbers suggests that
\[
\JCNN(\theta)\to \JC(\theta) := \frac{1}{2}g(\theta,\theta^\dagger) - \frac{1}{2}\log g(\theta,\theta^\dagger)
\]
almost surely. Observe that $\JC$ is minimized if and only if $g(\theta,\theta^\dagger) = 1$, which by \cref{ass:conv}(iv) occurs if and only if $\theta = \theta^\dagger$. We hence wish to establish convergence of the minimizers of $\JCNN$ to that of $\JC$.
In order to show this convergence, we use the approach of \cite{van1996weak}. Specifically we use the result of Exercise 3.2.3, which follows from Corollary 3.2.3(ii) and the Arzel\`a-Ascoli theorem. We must establish that:
\begin{enumerate}[(a)]
\item $\JCNN$ converges pointwise in probability to $\JC$;
\item the maps $\theta\mapsto \JCNN(\theta)$ are Lipschitz on $\Theta$ for each $N$, with (random) Lipschitz coefficients uniformly bounded in $N$ almost surely;
\item $\JC$ is lower semicontinuous with a unique minimum at $\theta^\dagger$; and
\item $\thetaCN = \mathcal{O}_P(1)$.
\end{enumerate}
The point (c) is true by assumption, and (d) follows since $\Theta$ is compact. To establish that point (a) holds, we note that it suffices to show that, in probability, for each $\theta \in \Theta$,
\begin{align}
\label{eqn:conv1}
\bigg|\frac{1}{2N}\sum_{j=1}^N b_j^N(\theta)(\zeta^2-1)\bigg| &\to 0\\
\label{eqn:conv2}
\bigg|\frac{1}{2N}\sum_{j=1}^N \left(b_j^N(\theta) - \log\frac{\mu_j(\theta^\dagger)}{\mu_j(\theta)}\right) - \frac{1}{N}\log\rho(\theta) - \JC(\theta)\bigg| &\to 0.
\end{align}
Note that the expression \cref{eqn:conv2} is deterministic. Define the map
\[
G_N(\theta) = \frac{1}{2N}\sum_{j=1}^N b_j^N(\theta)(\zeta_j^2-1).
\]
We show that $G_N(\theta) \to 0$ weakly for all $\theta \in \Theta$; since the limit is constant, the convergence then also occurs in probability. Combining \cref{lem:s_conv} with \cref{ass:conv}(ii),(iii), we see that
\begin{align}
\label{eq:conv_sum}
\frac{1}{N}\sum_{j=1}^N b_j^N(\theta) \to g(\theta,\theta^\dagger)
\end{align}
for each $\theta \in \Theta$. The proof of \cref{lem:s_conv} implies, in particular, that 
the sequence $\{b_j^N(\theta)\}_{j,N}$ is uniformly bounded for each $\theta$. Since $\zeta_j^2 \iid \chi_1^2$, we have that the characteristic function of $G_N(\theta)$ satisfies\footnote{Here $\log$ refers to the principal branch of the complex logarithm -- note that we are bounded away from the branch cut since the argument always has real part 1.}
\begingroup
\allowdisplaybreaks 
\begin{align*}
\mathbb{E}\Big(\exp\big(itG_N(&\theta)\big)\Big) = \prod_{j=1}^N \mathbb{E}\left(\exp\left(it\cdot\frac{1}{2N} b_j^N(\theta)(\zeta_j^2-1)\right)\right)\\
&= \prod_{j=1}^N \left(1 - \frac{b_j^N(\theta)it}{N}\right)^{-\frac{1}{2}}\exp\left(-it\cdot\frac{1}{2N} b_j^N(\theta)\right)\\
&= \exp\left(-\frac{1}{2}\sum_{j=1}^N \left[\log\left(1 - \frac{b_j^N(\theta)it}{N}\right) + \frac{b_j^N(\theta)it}{N}\right]\right)\\
&= \exp\left(-\frac{1}{2}\sum_{j=1}^N \left[-\frac{b_j^N(\theta)it}{N} - \frac{1}{2
}\left(\frac{b_j^N(\theta)it}{N}\right)^2 - \mathcal{O}(N^{-3}) + \frac{b_j^N(\theta)it}{N}\right]\right)\\
&= \exp\left(-\frac{1}{4}\sum_{j=1}^N \left[\frac{b_j^N(\theta)^2t^2}{N^2} - \mathcal{O}(N^{-3})\right]\right).
\end{align*}
\endgroup
From the boundedness of $\{b_j^N(\theta)\}_{j,N}$, we deduce that the sum in the exponent tends to zero as $N\to\infty$. It follows that
\[
\mathbb{E}\left(\exp\left(itG_N(\theta)\right)\right) \to \exp(0) = \mathbb{E}^{Z\sim\delta_0}\left(\exp(itZ)\right)
\]
and so $G_N(\theta) \to 0$ weakly; the convergence \cref{eqn:conv1} follows. We now rewrite the expression in \cref{eqn:conv2} as 
\begin{align*}
\frac{1}{2N}&\sum_{j=1}^N \left(b_j^N(\theta) - \log\frac{\mu_j(\theta^\dagger)}{\mu_j(\theta)}\right) - \frac{1}{N}\log\rho_0(\theta) - \JC(\theta)\\
&= \frac{1}{2N} \sum_{j=1}^N \left(b_j^N(\theta) - g(\theta,\theta^\dagger)\right) - \frac{1}{2N}\sum_{j=1}^N \left(\log\frac{\mu_j(\theta^\dagger)}{\mu_j(\theta)} - \log g(\theta,\theta^\dagger)\right) - \frac{1}{N}\log\rho_0(\theta).
\end{align*}
The first sum vanishes as $N\to\infty$ due to the convergence \cref{eq:conv_sum}, the second vanishes due to \cref{ass:conv}(ii), and third clearly vanishes. The convergence \cref{eqn:conv2} follows, and hence so does the pointwise convergence in probability $\JCNN\to\JC$. It remains to show the Lipschitz condition (b). We have, for any $\theta_1,\theta_2 \in \Theta$,
\begin{align*}
|\JCNN(\theta_1) - \JCNN&(\theta_2)| \leq \frac{1}{2N} \sum_{j=1}^N |b_j^N(\theta_1) - b_j^N(\theta_2)|\zeta_j^2\\
&\hspace{0.3cm} + \frac{1}{2N}\sum_{j=1}^N|\log\mu_j(\theta_1) - \log\mu_j(\theta_2)| + \frac{1}{2}|\log\rho_0(\theta_1) - \log\rho_0(\theta_2)|.
\end{align*}
By \cref{ass:conv}(v)-(vii) the Lipschitz property follows. The almost sure boundedness of the Lipschitz constants follows from the strong law of large numbers, since the i.i.d. random variables $\zeta_j^2$ have finite second moments.

In the case of $\JENN$, the limiting functional is the same: $\JE = \JC$. The proof for convergence of minimizers differs only in the expression \cref{eqn:conv2}, wherein the logarithmic term in the sum is replaced by $\log b_j^N(\theta)$; this does not affect the convergence of the expression. 

We now study (ii). The functional $\JNCNN$ differs from $\JCNN$ only in
the absence of the logarithmic term -- it is easy to see that the limiting functional is then given by
\[
\JNCN(\theta) := \frac{1}{2}g(\theta,\theta^\dagger),
\]
and that the required conditions (a)--(d) above are satisfied, since existence of a unique minimizer $\theta_*$ of $g(\cdot,\theta^\dagger)$ is assumed. The same result from \cite{van1996weak} may then be used to obtain the stated result.
\end{proof}

\begin{remark}
\label{rem:equiv}
An important implication of this result is that the hyperparameters can only be determined up to measure equivalence. By the Feldman--H\'ajek theorem, the measures $N(0,\mC(\theta^\dagger))$ and $N(0,\mC(\theta))$ are equivalent if and only if
\[
\sum_{j=1}^\infty\left(\frac{\mu_j(\theta^\dagger)}{\mu_j(\theta)} - 1\right)^2 < \infty
\]
which in particular implies that the limit $g(\theta,\theta^\dagger)$ is
identically $1$. The limiting functional is hence minimized by any $\theta$ that gives rise to an equivalent measure.
\end{remark}

\begin{remark}
\label{rem:rescale}
In some situations the limiting functional $g(\theta,\theta^\dagger)$ is
infinite whenever $\theta\neq\theta^\dagger$. Even though this limit 
clearly identifies the true hyperparameters, \cref{thm:main} does not directly apply, since, for example, \cref{ass:conv}(iii),(vii) cannot hold. One approach to avoid this is to replace the objective functional $\JCNN(\theta)$ by $\JCNN(\theta)^{\eps_N}$ for some positive sequence $\eps_N\to 0$ -- note that this does
not affect the sequence of minimizers since $t\mapsto t^{\eps_N}$ is strictly increasing for all $N$. Such a sequence $\{\eps_N\}$ may be chosen in practice to be such that $(\mu_N(\theta^\dagger)/\mu_N(\theta))^{\eps_N}$ converges to a finite value for each $\theta$ as $N\to\infty$. Examples of situations where these infinite limits occur, and appropriate choices of sequences $\{\eps_N\}$ to obtain finite limits, are discussed in what follows.
\end{remark}

\subsection{Examples}
\label{ssec:E}

We now provide examples which elucidate \cref{thm:main}.

\begin{example}[Whittle--Mat\'ern]
\label{ex:wm2}
Consider the case where the conditional Gaussian priors are Whittle--Mat\'ern distributions on a bounded domain $D\subseteq\R^d$. As mentioned in \cref{ex:wm}, the covariance operators diagonalize in the eigenbasis $\{\varphi_j\}$ of the Laplacian on $D$. Since $D$ is bounded we simply {\em define} the Whittle--Matern process to have covariance given by the inverse of \eqref{eq:prec}, and
where we equip the Laplacian with Dirichlet, Neumann or
periodic boundary conditions; we note that for all three such sets of
boundary conditions, the eigenvalues $\lambda_j$ of the negative Laplacian
tend to infinity. We first consider the case where we are hierarchical about the standard deviation $\sigma$ and the length-scale $\ell$, and denote $\theta = (\sigma,\ell) \in \Theta$. Fixing the regularity parameter $\nu>0$, the eigenvalues are given by
\begin{align*}
\mu_j(\theta) &= \kappa(\nu)\sigma^2 \ell^{d}(1+\ell^2\lambda_j)^{-\nu-d/2}\\
&=\kappa(\nu)\sigma^2 \ell^{-2\nu}(\ell^{-2}+\lambda_j)^{-\nu-d/2}
\end{align*}
for some constant $\kappa(\nu)$. We may then calculate
\[
g(\theta,\theta^\dagger) = \lim_{j\to\infty} \left(\frac{\sigma^\dagger}{\sigma}\right)^2\left(\frac{\ell}{\ell^\dagger}\right)^{2\nu}\left(1 + \frac{\ell^{-2}-(\ell^\dagger)^{-2}}{(\ell^\dagger)^{-2}+\lambda_j}\right)^{\nu} = \left(\frac{\sigma^\dagger}{\sigma}\right)^2\left(\frac{\ell}{\ell^\dagger}\right)^{2\nu}.
\]
We then see that $g(\theta,\theta^\dagger) = 1$ if and only if\footnote{This condition is slightly weaker than that required for measure equivalence -- for the measures to be equivalent we require in addition that $d \leq 3$, see for example Theorem 1 in \cite{DIS16}.} $\sigma\ell^{-\nu} = \sigma^\dagger(\ell^\dagger)^{-\nu}$. This equality is satisfied by infinitely many pairs $(\sigma,\ell)$. In order to apply \cref{thm:main} we require that the equality is only 
satisfied by the true hyperparameters. Therefore, instead of attempting to infer the pair $(\sigma,\ell)$, we attempt to infer the pair $(\sigma,\beta) := (\sigma,\sigma\ell^{-\nu})$; this is closely
related to the discussion around the Ornstein--Uhlenbeck process in 
\cref{ex:OU}. We then have
\[
\mu_j(\theta) = \kappa(\nu)\beta^2\left(\left(\frac{\beta}{\sigma}\right)^{2/\nu}+\lambda_j\right)^{-\nu-d/2}
\]
which leads to
\[
g(\theta,\theta^\dagger) = \left(\frac{\beta^\dagger}{\beta}\right)^2.
\]
When $\sigma$ is fixed, by applying \cref{thm:main}, we can deduce that the parameter $\beta$ is identifiable using via the centred MAP and empirical Bayesian methods; the proof that the requisite assumptions are satisfied under appropriate conditions is provided in \cref{lem:wm}. In particular, assuming the algebraic decay $a_j \asymp j^{-a}$ and $\gamma_N \asymp N^{-w}$, \cref{ass:conv}(ii) is equivalent to
\begin{align}
\label{eq:wm_cond}
w > a + \frac{\nu}{d} + \frac{1}{2}.
\end{align}
We also see that the parameter $\beta$ is not identifiable via the noncentred MAP method, since $g(\cdot;\theta)$ is minimized by taking $\beta$ as large as possible. 

In the case where we are hierarchical about the regularity parameter $\nu$, the assumptions of \cref{thm:main} do not hold. Nonetheless, the limiting 
functional can still be formally calculated as
\[
\JC(\nu) = 
\begin{cases}
\infty & \nu \neq \nu^\dagger\\
1 & \nu = \nu^\dagger
\end{cases}
\]
which is clearly minimized if and only if $\nu = \nu^\dagger$. As discussed 
in \cref{rem:rescale}, we can rescale to obtain a finite limiting functional; 
in this case making the choice $\eps_N = 1/\log(1+\lambda_N)$ achieves this.
\end{example}

\begin{example}[Automatic Relevance Determination]
\label{ex:ARD2}
The Automatic Relevance Determination (ARD) kernel is typically defined by
\[
c(x,x';\theta) = \sigma^2\exp\left(-\frac{1}{2}\sum_{k=1}^d\left(\frac{x_k-x_k'}{\theta_k}\right)^2\right).
\]
This is the Green's function for the anisotropic heat equation at time $t = 1$:
\[
\frac{\partial u}{\partial t}(t,x) = \sum_{k=1}^d \theta_k^2\frac{\partial u^2}{\partial x_k^2}(t,x),\quad u(0,x) = \sigma^2\xi(x).
\]
The corresponding covariance operator is hence given by
\[
\mC(\theta) = \sigma^2\exp(\Delta_\theta) := \sigma^2\exp\left(-\sum_{k=1}^d \theta_k^2\frac{\partial^2}{\partial x_k^2}\right).
\]
On rectangular domains this family of operators is simultaneously diagonalizable under the Laplacian eigenbasis. For example, if $D = (0,1)^d$ and we impose
Dirichlet boundary conditions on the Laplacian, then the eigenvalues are given by
\[
\mu_{i_1,\ldots,i_d}(\theta) = \sigma^2\exp\left(-\pi^2\sum_{k=1}^d \theta_k^2i_k^2\right).
\]
The results we have concerning consistency are given in terms of eigenvalues indexed by a single index $j$ rather than a multi-index $(i_1,\ldots,i_d)$. Rather than consider a particular enumeration of the multi-indices, we instead aim to infer each hyperparameter $\theta_k$ individually by only sending $i_k\to\infty$ -- this amounts to taking a subset of the observations. The problem of inferring each $\theta_k$ is then essentially equivalent to inference of the length-scale parameter of squared exponential prior with $d=1$. 

Note that \cref{thm:main} does not apply in this case -- the limiting functional $\JC$ is infinite everywhere except for the true hyperparameter, as was the case when inferring the parameter $\nu$ in the previous example. Again, following \cref{rem:rescale}, we can rescale to obtain a finite objective function;
in this case making the choice $\eps_N = 1/N^2$ suffices. ARD versions of general Whittle--Mat\'ern covariances can also be obtained by replacing the negative Laplacian $-\Delta$ with its anistropic 
analogue $-\Delta_\theta$ within the precision operator. It can be verified that the requisite assumptions for \cref{thm:main} are satisfied in this case when $\nu < \infty$; the proof is almost identical to that of \cref{lem:wm} and is hence omitted for brevity.
\end{example}

\section{Numerical Experiments}
\label{sec:NE}
In this section we present a number of numerical experiments in order to both validate the theory presented, and illustrate how the theory may extend beyond what has been proven.
\Cref{ssec:deblur} introduces a diagonalizable deblurring problem which is considered in the subsequent subsections.
\Cref{ssec:trunc} looks at the behaviour of minimizers of $\JCN$ with and without the prior truncation, as discussed in \cref{rem:trunc}.
\Cref{ssec:rates} looks at the traces of the errors between the hyperparameter estimates, comparing the convergence rates between the different functionals.
\Cref{ssec:equivalent} considers the setup of \cref{ex:wm2}, wherein the variance and length-scale parameters are to be jointly inferred; the minimizers are confirmed numerically to lie on the curve of hyperparameters which give rise to equivalent measures. Finally, \cref{ssec:noise_decay} considers  settings that enable us to test whether the assumptions of the theory are sharp -- in particular we see that they appear sharp only for the centred MAP approach, with the empirical Bayes estimates appearing to be more robust with respect to noise.

\subsection{Deblurring Problem}
\label{ssec:deblur}
In this subsection we consider the case that the forward map is given by a linear blurring operator. Let $\{\varphi_j\}_{j=0}^\infty$ denote the cosine Fourier basis on $D = (0,1)$,
\[
\varphi_j(x) = \sqrt{2}\cos(\pi jx),
\]
and define $A:L^2(D)\to L^2(D)$ by
\[
\la Au, \varphi_j\ra =
\begin{cases}
j^{-2}\la u,\varphi_j\ra & j \geq 1\\
0 & j = 0.
\end{cases}
\]
Then the map $A$ may be viewed as the solution 
operator $f\mapsto u$ for the problem
\begin{align}
\label{eq:pde}
-\Delta u(x) = \pi^2 f(x) \quad\text{for $x \in D$},\quad u'(0) = u'(1) = 0,\quad \int_0^1 u(x)\,\dee x = 0.
\end{align}
It could equivalently viewed as a convolution operator, writing
\[
(Au)(x) = \int_0^1 G(x,x')u(x')\,\dee x'
\]
where $G(x,x')$ is the Green's function for the system \cref{eq:pde}. This choice of forward operator is convenient as it diaganalizes in the same basis as the Whittle--Mat\'ern covariance operators on $D$, which are what we use throughout this subsection. In \cref{fig:truth_blur} we show the true state $u^\dagger$ that we fix throughout this subsection, and its image $Au^\dagger$ under $A$. It is drawn from a Whittle--Mat\'ern distribution with parameters $\sigma^\dagger = \ell^\dagger = 1$, $\nu^\dagger = 3/2$. To be explicit, in the notation of \cref{sec:consistency}, we have
\[
a_j = 1/j^2,\quad \mu_j(\theta) = \sigma^2\ell^{-2\nu}(\pi^2 j^2 + \ell^{-2})^{-\nu}.
\]

\subsection{Prior Truncation}
\label{ssec:trunc}
We first provide some numerical justification for the truncation of the prior at the same level as the observations when using the centred parameterization, as discussed in \cref{rem:trunc}. We fix a maximum discretization level $N_{\max} = 10^5$, and look at the behaviour of minimizers of the two functionals
\begin{align*}
\JCNN(\theta) &\eqc \frac{1}{2}\sum_{j=1}^N \frac{(y_j^\gamma)^2}{s^\gamma_j(\theta)} -  \frac{1}{2}\sum_{j=1}^N \log \frac{\mu_j(\theta^\dagger)}{\mu_j(\theta)} - \log \rho_0(\theta),\\
\widetilde{\JCNN}(\theta) &\eqc \frac{1}{2} \sum_{j=1}^N \frac{(y_j^\gamma)^2}{s^\gamma_j(\theta)} -  \frac{1}{2}\sum_{j=1}^{N_{\max}} \log \frac{\mu_j(\theta^\dagger)}{\mu_j(\theta)} - \log \rho_0(\theta),
\end{align*}
as $N$ is increased. We consider a conditional Whittle--Mat\'ern prior, treating the inverse length-scale $\theta = \ell^{-1}$ as a hyperparameter, and set $\gN = 1/N^5$ so that \cref{eq:wm_cond} is satisfied. In \cref{fig:prior_trunc} we show how the errors between the estimated inverse length-scales and the truth compare between the two functionals as $N$ increases. It can be seen that the error for the truncated prior is bounded above by that for the full prior, as expected.

\subsection{Centred, Noncentred and Empirical Bayes}
\label{ssec:rates}
We now compare numerically the behaviour of optimizers of the three functionals $\JCNN,\JNCNN$ and $\JENN$, and verify that the conclusions of \cref{thm:main} hold. As above, we consider a conditional Whittle--Mat\'ern prior with the inverse length-scale $\theta = \ell^{-1}$ as a hyperparameter, and set $\gN = 1/N^5$. In \cref{fig:rates} we show how the errors between the three sequences of minimizers and the truth compare as $N$ increases. We see that the noncentred MAP error diverges, as expected: the limiting functional is given by
\[
\JNC(\theta) = \frac{1}{2}\left(\frac{\ell}{\ell^\dagger}\right)^3,
\]
which is minimized as $\ell^{-1}\to\infty$. The empirical Bayes and centred MAP errors both generally decrease as $N$ is increased, again as expected, with the empirical Bayes estimate slightly outperforming the centred MAP estimate for moderate $N$; the noncentred MAP estimator fails to converge.

Also in \cref{fig:rates} we show the same errors averaged over 1000 independent realizations of the truth $u^\dagger \sim N(0,\mC(\theta^\dagger))$ and noise $\eta \sim N(0,\mI)$, and the same behaviour is observed. A reason for the empirical Bayes estimate outperforming the noncentred MAP estimate for moderate $N$ may be that the terms in the summation in the functional \cref{eq:jen} taking the form $x_j - \log x_j$, rather than $x_j - \log x_j^\eps$ for some $x_j^\eps \approx x_j$ as in \cref{eq:jcn}, which is minimized by $x_j = 1$. For larger $N$ there is very little difference between the two functionals, since $\gN\to 0$.

\subsection{Equivalent Families of Measures}
\label{ssec:equivalent}
We now consider the same setup as the previous subsubsection, but treat both the inverse length-scale $\ell^{-1}$ and the standard deviation $\sigma$ as hyperparameters: $\theta = (\sigma,\ell^{-1})$. The resulting family of conditional prior measures are then equivalent along any curve $\{(\sigma,\ell^{-1})\,|\,\sigma\ell^{-\nu} = \mathrm{constant}\}$, as discussed in \cref{ex:wm2}, and so the hyperparameters cannot be identified beyond this curve. This is illustrated in \cref{fig:matern_sig_ell}. In the top row we plot the functional $\JCNN(\sigma,\ell)$ for $(\sigma,\ell^{-1}) \in (0,5)^2$, with $N$ increasing from left to right. In the bottom row we plot the sets
\[
\left\{(\sigma,\ell^{-1})\,\bigg|\, \ell \in \underset{\ell^{-1} \in (0,5)}{\argmin}\,\JCNN(\sigma,\ell)\right\},\quad \left\{(\sigma,\ell^{-1})\,\bigg|\, \sigma \in \underset{\sigma \in (0,5)}{\argmin}\,\JCNN(\sigma,\ell)\right\},
\]
along with the curve $\sigma\ell^{-\nu} = \sigma^\dagger(\ell^\dagger)^{-\nu}$, i.e. $g(\cdot,\theta^\dagger)^{-1}(1)$; the global minimizer (i.e. the intersection of these sets) is also shown as a green dot. We see that the sets of minimizers concentrate on the limiting curve $g(\cdot,\theta^\dagger)^{-1}(1)$ as $N$ is increased.

For reference, we also consider the same experiments, but working with the reparameterization $\theta = (\sigma,\beta)$ introduced in \cref{ex:wm2} so that $\beta$ should be identifiable. In \cref{fig:matern_sig_beta} we see that this is indeed the case, with the curves now concentrating on the line $\beta = \sigma^\dagger(\ell^\dagger)^{-\nu} = 1$ as $N$ is increased.

\subsection{Noise Decay Rate}
\label{ssec:noise_decay}
We choose here now to be hierarchical about just the inverse length-scale $\theta = \ell^{-1}$. In the theory we made the assumption \cref{ass:conv}(ii) concerning the decay rate of the forward map and covariance singular values versus the decay of the noise level. For Whittle--Mat\'ern priors, assuming the algebraic decay $a_j\asymp j^{-a}$ and $\gamma_N\asymp N^{-w}$, the required condition on $w$ for \cref{ass:conv}(ii) to hold is given by \cref{eq:wm_cond}. In the setup considered here, this translates to $w > 4$. We now investigate numerically whether this condition is sharp, making the three choices $\gN = N^{-w}$ for $w=3.5,4,4.5$. The resulting error traces are shown in \cref{fig:gammaN} for the centred MAP and empirical Bayesian methods. It appears that the condition is likely to be sharp for the centred optimization, given that convergence fails at the borderline case. However. For the empirical Bayesian optimization the condition does not appear to be necessary, with convergence occurring in all cases, suggesting it is a more stable estimator than the MAP.

In light of \cref{rem:Ngamma}, we also consider the same setup, but with a fixed noise level and increasing $N$. Making the choice $N_\gamma = \gamma^{-1/w}$, the condition on $w$, equivalent to \cref{eq:wm_cond2}, is the same as before. In \cref{fig:Ngamma} we show the errors for the choices $w=3.5.4,4.5$, and the same trends are observed as for \cref{fig:gammaN}.

\begin{figure}
\centering
\includegraphics[width=0.5\textwidth, trim=3cm 0cm 3cm 0cm, clip]{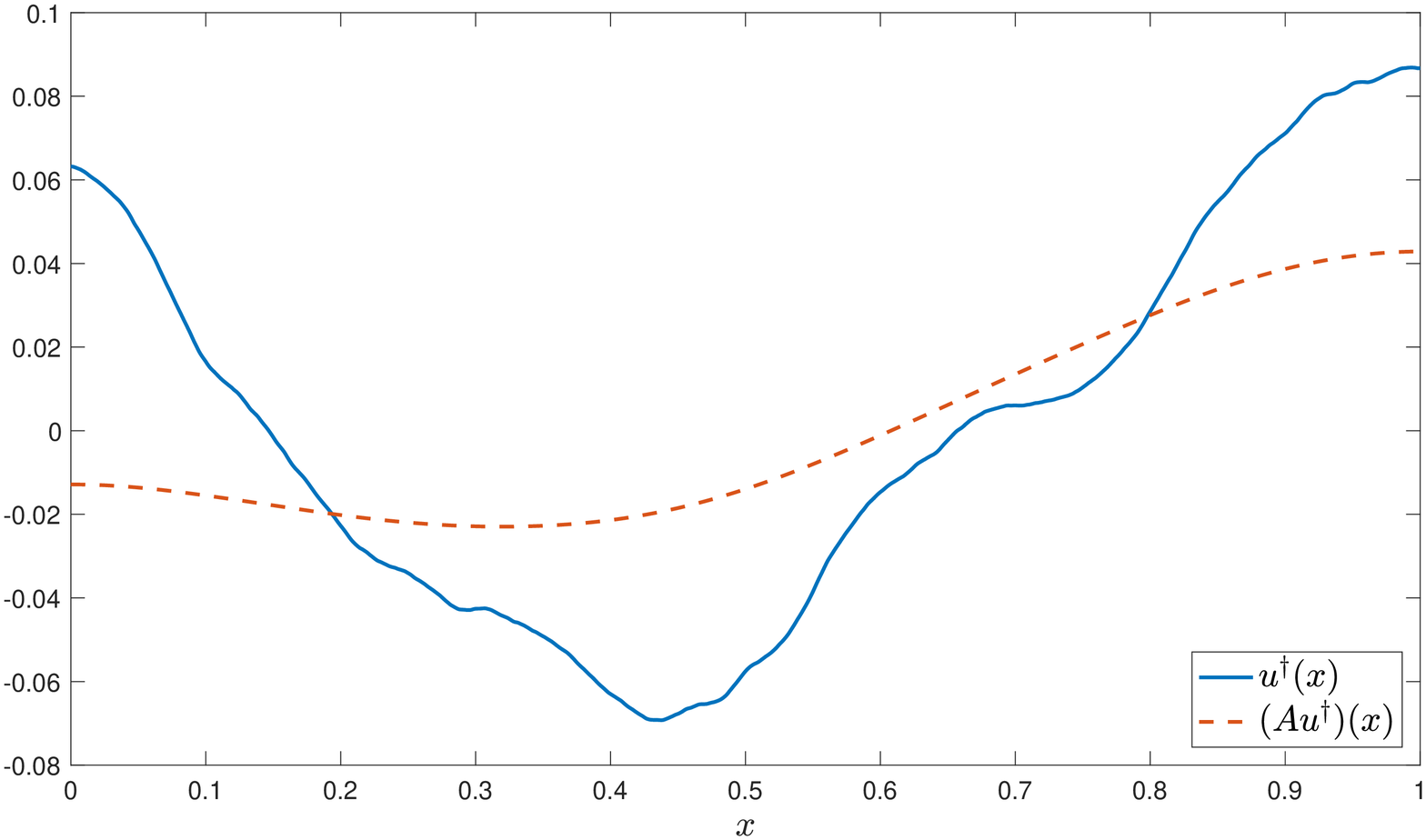}
\caption{The true state $u^\dagger$ used throughout \cref{ssec:deblur}, and its image $Au^\dagger$ under the blurring operator $A$.}
\label{fig:truth_blur}
\end{figure}

\begin{figure}
\centering
\includegraphics[width=0.55\textwidth, trim=1cm 0cm 1cm 0cm, clip]{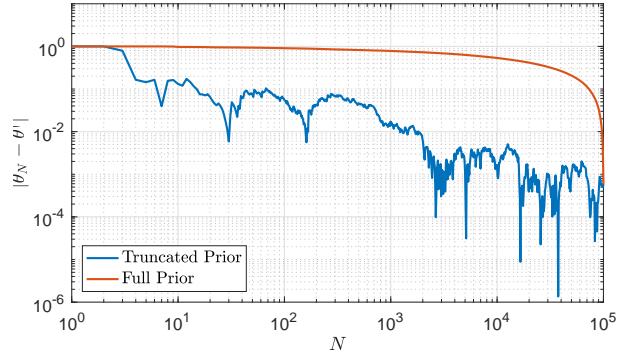}
\caption{The trace of the errors between the minimizers of $\JNCNN$ and $\widetilde{\JNCNN}$ and the true hyperparameter, as defined in \cref{ssec:trunc}, as $N$ is increased.}
\label{fig:prior_trunc}
\end{figure}

\begin{figure}
\centering
\includegraphics[width=\textwidth, trim=2cm 0cm 3.8cm 0cm, clip]{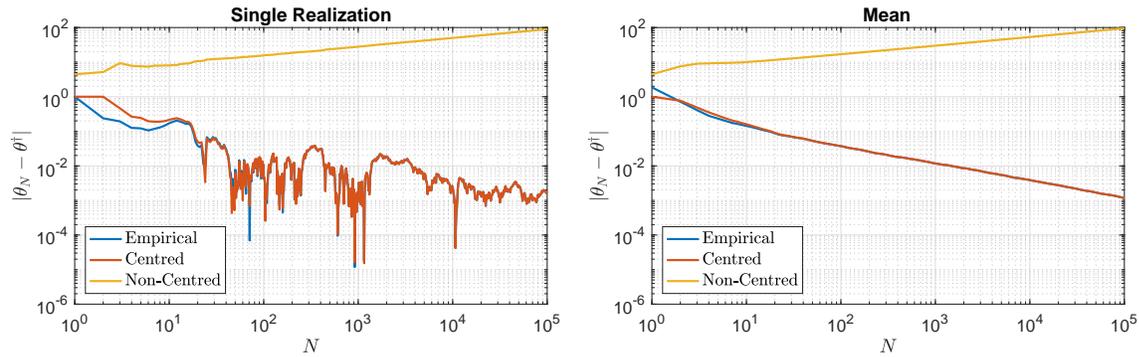}
\caption{Comparison of the errors between the minimizers of the three functionals $\JENN,\JCNN,\JNCNN$ and the true hyperparameter, as $N$ is increased. The left figure shows the error traces for a single realization of the truth and the noise, and the right figure shows the errors averaged over 1000 such realizations.}
\label{fig:rates}
\end{figure}

\begin{figure}
\centering
\includegraphics[width=\textwidth, trim=3cm 0cm 3cm 0cm, clip]{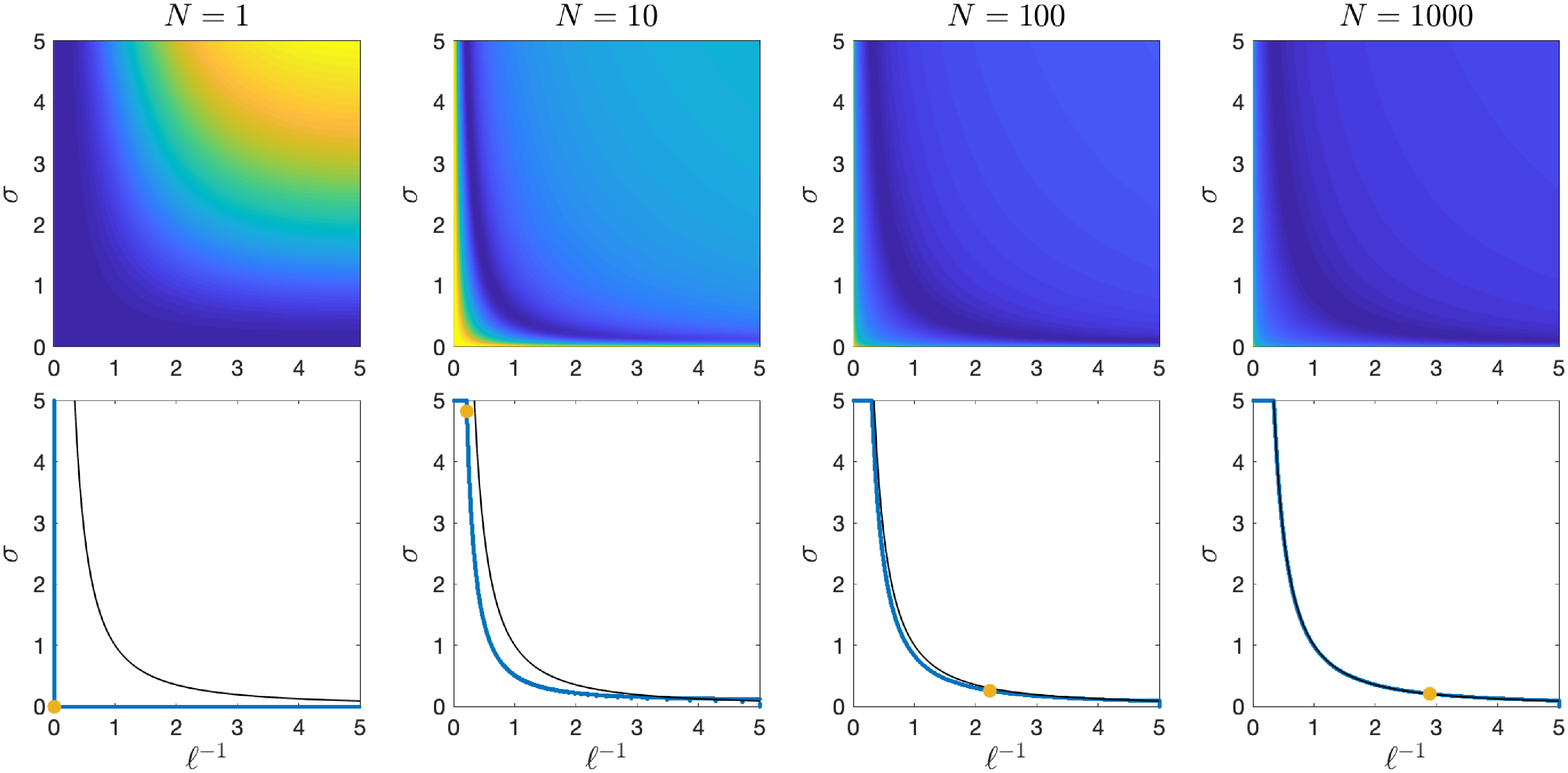}
\caption{(Top) The objective function $\JCNN(\sigma,\ell)$ for $N=1,10,100,1000$. (Bottom) The locations of the minimizers of each $\JCNN(\sigma,\ell)$ across each row and column of the computed approximations (blue), the curve of parameters that produce equivalent measures to the true parameter (black), and the global optimizer (green).}
\label{fig:matern_sig_ell}
\end{figure}

\begin{figure}
\centering
\includegraphics[width=\textwidth, trim=3cm 0cm 3cm 0cm, clip]{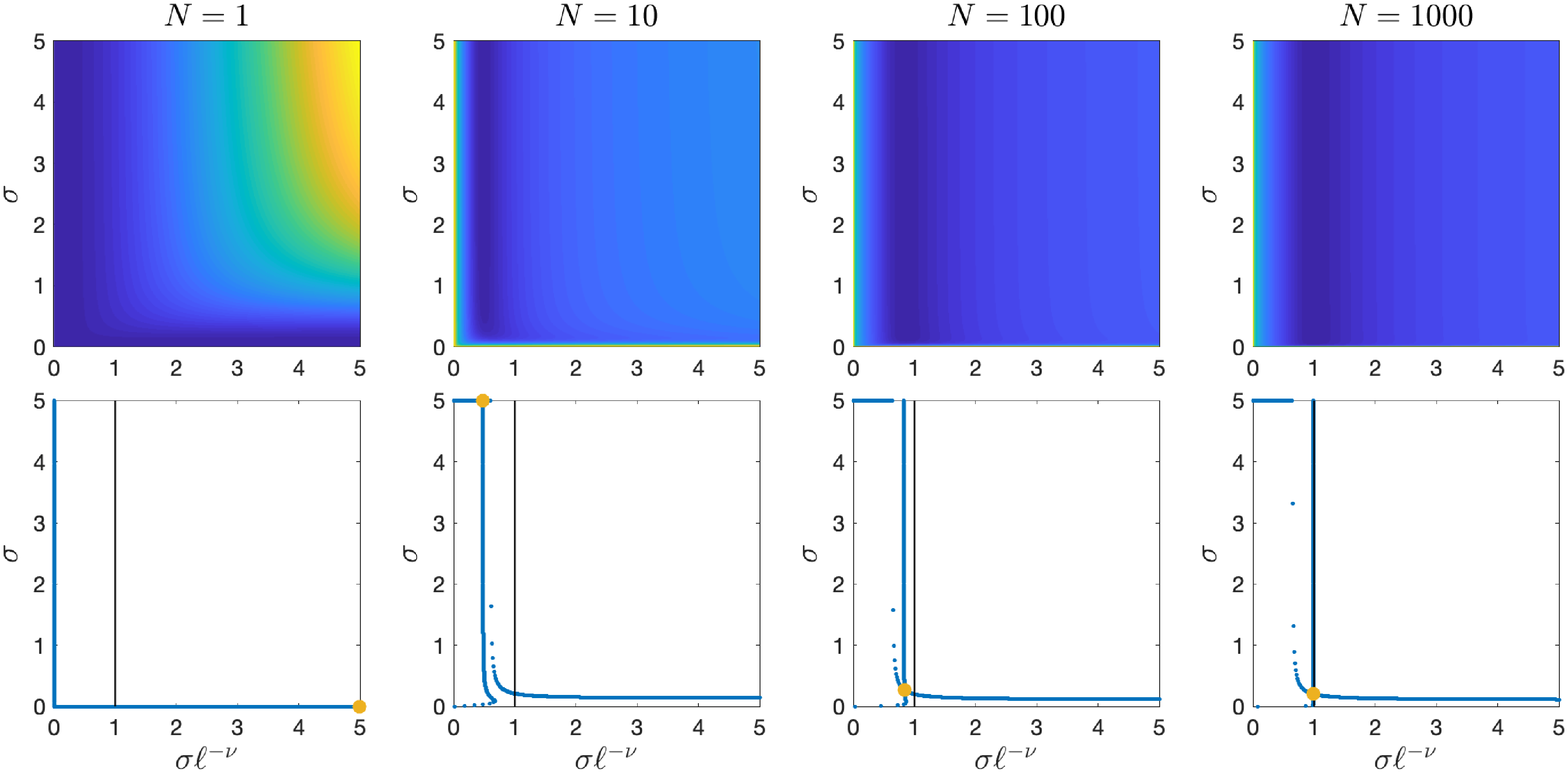}
\caption{(Top) The objective function $\JCNN(\sigma,\beta)$ for $N=1,10,100,1000$. (Bottom) The locations of the minimizers of each $\JCNN(\sigma,\beta)$ across each row and column of the computed approximations (blue), the curve of parameters that produce equivalent measures to the true parameter (black), and the global optimizer (green).}
\label{fig:matern_sig_beta}
\end{figure}

\begin{figure}
\centering
\includegraphics[width=\textwidth, trim=2cm 0cm 3.8cm 0cm, clip]{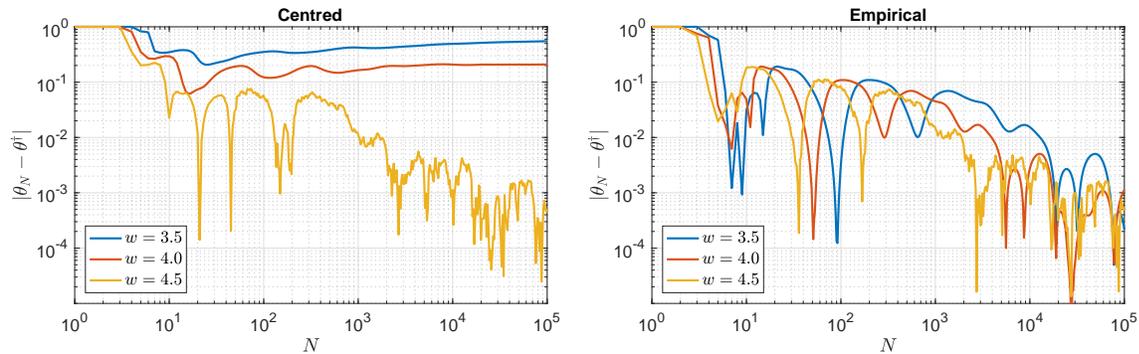}
\caption{Traces of the errors between optimizers of $\JCNN(\theta)$ (left), $\JENN(\theta)$ (right) and the true hyperparameter, as $N$ is increased. Here the noise level $\gamma_N$ is taken as $\gamma_N = N^{-w}$ for $w=3.5,4,4.5$.}
\label{fig:gammaN}
\end{figure}

\begin{figure}
\centering
\includegraphics[width=\textwidth, trim=2cm 0cm 3.8cm 0cm, clip]{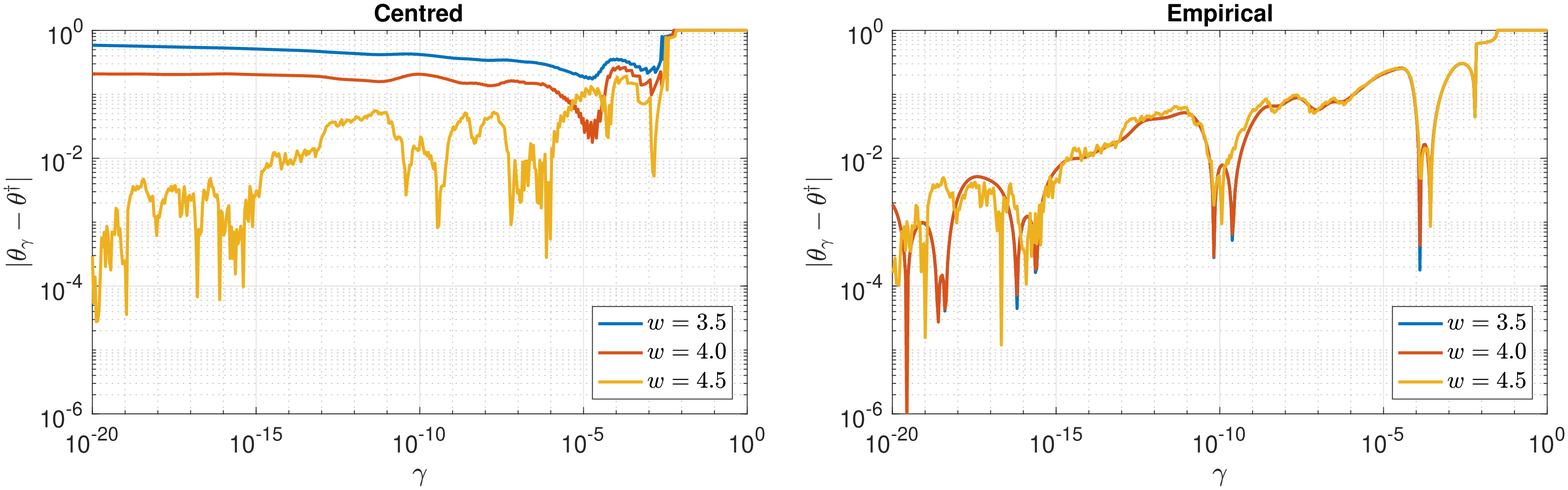}
\caption{Traces of the errors between optimizers of $\mathsf{J}_{\mathsf{C}}^{N_\gamma,\gamma}(\theta)$ (left), $\mathsf{J}_{\mathsf{E}}^{N_\gamma,\gamma}(\theta)$ (right) and the true hyperparameter, as $\gamma$ is decreased. Here the number of observations $N_\gamma$ is taken as $N_\gamma = \gamma^{-1/w}$ for $w=3.5,4,4.5$.}
\label{fig:Ngamma}
\end{figure}

\section{Conclusions}
\label{sec:C}

Learning hyperparameters in Bayesian hierarchical inference is important
in two main contexts: when the hyperparameters themselves are the primary
object of inference, and the underlying quantity which depends on them
{\em a priori} is viewed as a nuisance parameter; when the hyperparameters
themselves are not of direct interest, but choosing them carefully 
aids in inferring the underlying quantity which depends on them
{\em a priori}. In both settings it is of interest to understand when 
hyperparameters can be accurately inferred from data. In this paper
we have studied this question within the context of MAP estimation.
Our work suggests the benefits of using the centred parameterization
over the noncentred one, and also supports the use of empirical Bayes
procedures. This is interesting because the relative merits of  
centring and noncentring in this context differ from what is found
for sampling methods such as MCMC.

The theorem is confined to a straightforward situation, concerning
linear inverse problems, in which the
relevant operators are simultaneously diagonalizable. It also 
imposes conditions on the the parameters defining the problem;
numerical experiments indicate that these are sharp for the
centred MAP estimator, but not for the empirical Bayes
estimator, demonstrating that the latter is preferable.
It would also be
of interest to push the boundaries of the theory outside this regime
to the non-diagonal setting and even into nonlinear inverse problems.
It would also be of interest to study fully Bayesian posterior
inference for the hyperparameters, and Bernstein-von Mises theorems;
this may be related the re-scalings needed at the end of \Cref{ex:wm2,ex:ARD2}.

\vspace{0.1in}

\noindent{\bf Acknowledgements} The work of AMS and MMD is funded by
US National Science Foundation (NSF) grant DMS 1818977 and
AFOSR Grant FA9550-17-1-0185.

\bibliographystyle{siamplain}
\bibliography{main}

\pagebreak
\appendix
\section{Supporting Lemmas}
In this appendix we provide a number of lemmas that are used during proofs and examples in the main text.

\begin{lemma}
\label{lem:determinant}
Let $m \geq n$, $A \in \R^{m\times n}$ and $Q \in \R^{m\times m}$. Then
\[
\det(A^*QA) = \det(Q)\det(AA^*).
\]
\end{lemma}

\begin{proof}
Let $A = U\Sigma V^*$ be the singular value decomposition of $A$, with $U \in \R^{m\times m}$, $V \in \R^{n\times n}$ unitary, and $\Sigma \in \R^{m\times n}$. Then we have
\[
\det(A^*QA) = \det(V\Sigma^*U^*QU\Sigma V^*) = \det(\Sigma^*U^*QU\Sigma)\det(V^*V) = \det(\Sigma^*U^*QU\Sigma).
\]
We have that
\begin{align*}
(\Sigma^*U^*QU\Sigma)_{ij} &= 
\begin{cases}
\Sigma_{ii}(U^*QU)_{ij}\Sigma_{jj} & i,j \leq m\\
0 & i> m\text{ or } j > m
\end{cases}\\
&= (\hat{\Sigma}U^*QU\hat{\Sigma})_{ij}
\end{align*}
where $\hat{\Sigma} \in \R^{m\times m}$ is given by $\hat{\Sigma}_{ij} = \Sigma_{ij}$. Since all matrices are now square, we see that
\begin{align*}
\det(\Sigma^*U^*QU\Sigma) &= \det(\hat\Sigma U^*QU\hat\Sigma)\\
&= \det(Q)\det(U^*U)\det(\hat\Sigma^2)\\
&= \det(Q)\det(\Sigma\Sigma^*)\\
&= \det(Q)\det(AA^*).
\end{align*}
\end{proof}

\begin{lemma}
\label{lem:s_conv}
Let $\{a_j\},\{\mu_j\},\{\bar\mu_j\}$ and $\{\gamma_j\}$ be positive sequences with $\bar\mu_j/\mu_j \to g \geq 0$. Then
if 
\[
\min_{j=1,\ldots,N} \frac{a_j^2\mu_j}{\gamma_N^2} \to \infty\quad\text{as }N\to\infty
\]
we have
\[
\frac{1}{N}\sum_{j=1}^N \frac{a_j^2\bar\mu_j + \gamma_N^2}{a_j^2\mu_j + \gamma_N^2} \to g\quad\text{as }N\to\infty.
\]
\end{lemma}

\begin{proof}
We write
\begin{align*}
\frac{a_j^2\bar\mu_j + \gamma_N^2}{a_j^2\mu_j + \gamma_N^2} &= \frac{\bar\mu_j + \gamma_N^2/a_j^2}{\mu_j + \gamma_N^2/a_j^2}\\
&= \frac{\bar\mu_j}{\mu_j} + \left(\frac{\bar\mu_j + \gamma_N^2/a_j^2}{\mu_j + \gamma_N^2/a_j^2}- \frac{\bar\mu_j}{\mu_j}\right)\\
&= \frac{\bar\mu_j}{\mu_j} + \frac{\gamma_N^2/a_j^2(\mu_j-\bar\mu_j)}{\mu_j^2 + \gamma_N^2/a_j^2\mu_j}\\
&= \frac{\bar\mu_j}{\mu_j} + \frac{1}{a_j^2\mu_j/\gamma_N^2 + 1}\left(1-\frac{\bar\mu_j}{\mu_j}\right).
\end{align*}
Now observe that
\begin{align*}
\left|\frac{1}{N}\sum_{j=1}^N \frac{a_j^2\bar\mu_j + \gamma_N^2}{a_j^2\mu_j + \gamma_N^2} - g\right| &\leq \left|\frac{1}{N}\sum_{j=1}^N \left(\frac{a_j^2\bar\mu_j + \gamma_N^2}{a_j^2\mu_j + \gamma_N^2} - \frac{\bar\mu_j}{\mu_j}\right)\right| + \left|\frac{1}{N}\sum_{j=1}^N\frac{\bar\mu_j}{\mu_j}-g\right|.
\end{align*}
The second term on the right hand side tends to zero by the assumed convergence. From the above, the first term is equal to
\begin{align*}
\left|\frac{1}{N}\sum_{j=1}^N \frac{1}{a_j^2\mu_j/\gamma_N^2 + 1}\left(1-\frac{\bar\mu_j}{\mu_j}\right)\right| &\leq \max_{j=1,\ldots,N} \frac{1}{a_j^2\mu_j/\gamma_N^2 + 1}\left|1-\frac{\bar\mu_j}{\mu_j}\right|\\
&\leq C \left(\min_{j=1,\ldots,N}a_j^2\mu_j/\gamma_N^2 + 1\right)^{-1}
\end{align*}
again using the assumed convergence of the ratio $\bar\mu_j/\mu_j$; the result follows.

\end{proof}

In the following, given two sequence $\{a_j\},\{b_j\}$, we write $a_j\asymp b_j$ if there exist constants $c_1,c_2 > 0$ such that $c_1a_j \leq b_j \leq c_2a_j$ for all $j$.
\begin{lemma}
\label{lem:wm}
Let $\Theta = [\beta_-,\beta_+] \subseteq (0,\infty)$. Given $\nu,\sigma > 0$, $d \in \N$ and a positive sequence $\lambda_j \asymp j^{2/d}$ define
\[
\mu_j(\beta) = \beta^2\left(\left(\frac{\beta}{\sigma}\right)^{2/\nu} + \lambda_j\right)^{-\nu-d/2}.
\]
Assume that $a_j \asymp j^{-a}$ and $\gamma_N \asymp N^{-w}$, where $w,a>0$ are such that
\[
w > a + \frac{\nu}{d} + \frac{1}{2}.
\]
Then \cref{ass:conv}(i)-(vi) hold.
\end{lemma}

\begin{proof}
\begin{enumerate}[(i)]
\item This is true by assumption.
\item We assume without loss of generality that $a_j,\lambda_j$ are monotonically decreasing. Then 
\begin{align*}
\min_{j=1,\ldots,N}\frac{a_j^2\mu_j(\beta)}{\gamma_N^2} = \frac{a_N^2\mu_N(\beta)}{\gamma_N^2}.
\end{align*}
We may bound the right hand side as
\begin{align*}
\frac{a_N^2\mu_N(\beta)}{\gamma_N^2} \asymp N^{2(w-a)} \beta^2\left(\left(\frac{\beta}{\sigma}\right)^{2/\nu} + \lambda_N\right)^{-\nu-d/2} \asymp N^{2(w-a-\nu/d-1/2)},
\end{align*}
which diverges given the assumption on the parameters.
\item In \cref{ex:wm2} it is demonstrated that
\[
g(\beta,\beta^\dagger) = \lim_{j\to\infty} \frac{\mu_j(\beta^\dagger)}{\mu_j(\beta)} = \left(\frac{\beta^\dagger}{\beta}\right)^2
\]
for all $\beta \in \Theta$. The map $g(\beta,\beta^\dagger)-\log g(\beta,\beta^\dagger)$ is clearly continuous on $\Theta$, and so in particular lower semicontinuous. 
\item This is clearly true.
\item We have that
\[
\log\mu_j(\beta) = 2\log\beta - \left(\nu+\frac{d}{2}\right)\log\left(\left(\frac{\beta}{\sigma}\right)^{2/\nu} + \lambda_j\right)
\] 
which is smooth on $\Theta$, and so
\begin{align*}
\left|\frac{\dee}{\dee\beta}\log\mu_j(\beta)\right| &= \left|\frac{2}{\beta} - \left(\nu+\frac{d}{2}\right)\frac{2}{\nu}\left(\frac{\beta}{\sigma}\right)^{2/\nu-1}\left(\left(\frac{\beta}{\sigma}\right)^{2/\nu} + \lambda_j\right)^{-1}\right|\\
&\leq \frac{2}{\beta_-} + \left(\nu+\frac{d}{2}\right)\frac{2}{\nu}\frac{\sigma}{\beta_-}.\end{align*}
It follows that $\log\mu_j(\beta)$ is Lipschitz with Lipschitz constants bounded in $j$.
\item The map $b_j^N(\beta)$ is smooth on $\Theta$, and we have that
\[
|(b_j^N)'(\beta)| = \left|b_j^N(\beta)\frac{\mu_j'(\beta)}{\mu_j(\beta) + \gamma_N^2/a_j^2}\right| \leq \left|b_j^N(\beta)\right|\left|\frac{\mu_j'(\beta)}{\mu_j(\beta)}\right| =  \left|b_j^N(\beta)\right|\left|\frac{\dee}{\dee\beta}\log\mu_j(\beta)\right|,
\]
and the final term on the right hand side is uniformly bounded by part (v). Finally observe that
\[
|b_j^N(\theta)| \leq \frac{\mu_j(\beta^\dagger)}{\mu_j(\beta)} + \frac{\gamma_N^2}{a_j^2\mu_j(\beta)}.
\]
The first term can be seen to be uniformly bounded by noting that $\beta_->0$, and the second term by using part (ii). The map $\beta_j^N(\beta)$ is hence Lipschitz with Lipschitz constants bounded in $j,N$.
\end{enumerate}
\end{proof}

%
%

\end{document}